\documentclass[10pt,a4paper]{amsart}
\usepackage[toc,page]{appendix}
\usepackage{hyperref}
\usepackage{mathrsfs}
\usepackage{amsmath}
\usepackage{amssymb}
\usepackage[all]{xy}
\usepackage{latexsym}
\usepackage{ bbold }

\mathchardef\dashmod="2D

\usepackage{datetime}
\usepackage{pgf,tikz}
\usetikzlibrary{arrows}

\usepackage[latin1]{inputenc}
\usepackage[T1]{fontenc}

\renewcommand{\ker}{\mathrm{ker}}

\newcommand{\ext}{\mathrm{Ext}}

\renewcommand{\mod}{\mathbf{Mod}}
\newcommand{\cat}{\mathrm{Cat}}
\newcommand{\Top}{\mathcal{T}}
\newcommand{\sh}{\mathbf{SH}}
\newcommand{\stinf}{\mathbf{St}_{\infty}}
\newcommand{\st}{\mathbf{St}}
\newcommand{\un}{\mathbb{1}}
\newcommand{\eq}{\cong}
\newcommand{\steq}{\simeq}

\newcommand{\stext}{\underline{\mathrm{Ext}}}

\newcommand{\gm}{\mathrm{Aut}}
\newcommand{\End}{\mathrm{End}}
\newcommand{\pic}{\mathrm{pic}}
\newcommand{\Pic}{\mathrm{Pic}}

\newcommand{\Z}{\mathbb{Z}}
\newcommand{\F}{\mathbb{F}}

\newcommand{\E}{\mathcal{E}}
\newcommand{\A}{\mathcal{A}}
\newcommand{\B}{\mathcal{B}}

\newcommand{\C}{\mathcal{C}}
\newcommand{\D}{\mathcal{D}}
\renewcommand{\L}{\mathcal{L}}

\newcommand{\sur}{/ \hspace{-0.07cm}/}

\theoremstyle{definition}
\newtheorem{de}{Definition}[section]

\theoremstyle{plain}
\newtheorem{thm}[de]{Theorem}
\newtheorem{lemma}[de]{Lemma}
\newtheorem{pro}[de]{Proposition}
\newtheorem{cor}[de]{Corollary}

\newtheorem{prob}[de]{Problem}

\newtheorem*{thm*}{Theorem}
\newtheorem*{lemma*}{Lemma}
\newtheorem*{pro*}{Proposition}
\newtheorem*{cor*}{Corollary}
\newtheorem*{prob*}{Problem}

\theoremstyle{remark}
\newtheorem{rk}[de]{Remark}
\newtheorem{ex}[de]{Example}

\title[Picard Hopf]{The stable Picard group of Hopf algebras via descent, and an application}
\author{Nicolas Ricka}
\address{Department of Mathematics, Wayne State University \\
 Detroit, MI 48202}
\email{nicolas.ricka@wayne.edu}
\thanks{The author is indebted to Akhil Mathew for suggesting such an approach to the study of Picard groups of Hopf algebras}

\begin{document}
\begin{abstract}
Let $A$ be a cocommutative finite dimensional Hopf algebra over the field with two elements, satisfying some mild hypothesis. We set up a descent spectral sequence which computes the Picard group of the stable category of modules over $A$. The starting point is the observation that the stable category of $A$-modules can be reconstructed, as an $\infty$-category, as the totalization of a cosimplicial $\infty$-category whose layers are related to the stable categories of modules over the quasi-elementary sub-Hopf-algebras of $A$. This leads to a spectral sequence computing the Picard group which, in some cases, is completely understood. This also leads to a spectral sequence answering a lifting problem in the category of $A$-modules.
We then show how to apply this machinery to compute Picard groups and solve lifting problems in the case of $\A(1)$-modules, where $\A(1)$ is the subalgebra of the Steenrod algebra generated by the two first Steenrod squares.
\end{abstract}

\maketitle

\section*{Introduction}

\textbf{Conventions:} Let $\F$ be the field with two elements. Every algebraic structure is implicitly over the base field $\F$. The Hopf algebras under consideration in this paper are connected, cocommutative finite dimensional Hopf algebras, unless explicitly specified otherwise. Moreover, unless specified otherwise, $\ext$ means extension group in the stable category, so that $\ext^{s}$ is defined for all $s \in \Z$.

\textbf{Statement of the results.}

Let $A$ be a Hopf algebra. The diagonal $\Delta: A \rightarrow A \otimes A$ induces a monoidal product $\otimes$ on the category ${_A}\mod$ of graded $A$-modules refining the tensor product of $\F$-vector spaces. The unit for this monoidal structure is $\F$ concentrated in degree zero, denoted $\un$.
As $A$ is cocommutative, $({_A}\mod, \otimes, \un)$ is a symmetric monoidal category.

Let $\stinf(A)$ be the stable category of $A$-modules (see Proposition \ref{pro:stablecat}). The classical construction gives the stable category of modules as a model category. Here, we merely need its underlying $\infty$-category. We denote $\st(A)$ its associated homotopy category.

It has been clear for a long time that some sub-Hopf algebras of $A$ are of particular importance. These are the quasi-elementary sub-Hopf-algebras, see Definition \ref{de:quasielem}. The definition being quite technical for this introduction, the reader in encouraged to think of these as the analogue of the group algebras for elementary abelian subgroups in representation theory of finite groups.

\begin{pro}[Theorem 1.2-1.4, \cite{Pal97}]
The quasi-elementary sub-Hopf algebras of $A$ detects the following properties: \begin{itemize}
\item Nilpotence in $\ext(\un, \un)$, and in $\ext(M,M)$, for any module $M$,
\item Freeness of a module.
\end{itemize}
Precisely, an element in $\ext_A(M,M)$ is nilpotent if and only if all of its restrictions in $\ext_E(M,M)$ are nilpotent, for all quasi-elementary sub-Hopf algebra $E$ of $A$, and an $A$-module $M$ is free if and only if all of its restrictions to quasi-elementary sub-Hopf algebras of $A$ are free.
Moreover, there is a Quillen $F$-isomorphism (isomorphism modulo nilpotents)
$$  \ext_A(\un, \un) \eq \lim_B \ext_B(\un, \un),$$
where the limit is taken over all the quasi-elementary sub-Hopf-algebras of $A$.
\end{pro}

However, finding all the quasi-elementary sub-Hopf algebras of a given Hopf algebra $A$ is a difficult problem in general. Hopefully, for a large class of algebras, these are completely identified with a the simpler notion of exterior Hopf subalgebras.

\begin{ex}
For any sub-Hopf-algebra $A$ of the modulo $2$ Steenrod algebra, the quasi-elementary Hopf subalgebras of $A$ are exactly the exterior Hopf subalgebras by \cite[Section 2.1.1]{Pal01}. Moreover, these are completely identified in \textit{loc cit}.
For a finite group $G$, one can consider the associated group Hopf-algebra $\F G$. In that case, the elementary sub-Hopf algebras of $\F G$ correspond to elementary abelian subgroups of $G$.
\end{ex}

Let $A$ be a cocommutative Hopf algebra, for example a sub-Hopf-algebra of the Steenrod algebra. The classification of all indecomposable modules is known to be tricky in the simplest cases (see for example the classification of indecomposable $\A(1)$-modules given by Crawley-Boevey in \cite{CB89}), and that bigger Hopf algebras are wild.

A first invariant of the stable module category, that is both still interesting and computable is the Picard group of this category, that is the group of stably invertible modules. These modules, called endotrivial modules in representation theory, played a fundamental role in the study of representations of finite $p$-groups, and motivated the work \cite{CT04,CT05} which offers a complete classification of the Picard group of the stable category of representations of finite $p$-groups. \\

The starting point of our analysis is the following result, which says roughly that the stable category of $A$-modules can be recovered from the stable categories of $E$-modules, where $E$ runs through the quasi-elementary sub-Hopf-algebras of $A$.

\begin{thm*}[{Theorem \ref{thm:totalization}}]
The functor $T \otimes : \stinf(A) \rightarrow \mod_{\stinf(A)}(T)$
extends to an equivalence of stable monoidal $\infty$-categories
\begin{equation}
 \stinf(A) \cong Tot\left(  \xymatrix{\mod_{\stinf(A)}(T)
  \ar@<-.5ex>[r] \ar@<.5ex>[r] & \mod_{\stinf(A)}(T\otimes T) \ar@<-.7ex>[r] \ar@<-.0ex>[r] \ar@<.7ex>[r]
& \hdots}\right),
\end{equation}
where $T$ is an algebra object in $\stinf(A)$ related to the forgetful functors from $\stinf(A)$ and $\stinf(E)$, for $E \subset A$ (see Definition \ref{de:defT}).
When $A$ admits only one elementary sub-Hopf-algebra $E$, this reduces to
\begin{equation*} 
 \stinf(A) \cong Tot\left(  \xymatrix{\stinf(E) \ar@<-.5ex>[r] \ar@<.5ex>[r] & \mod_{\stinf(E)}(T_E) \ar@<-.7ex>[r] \ar@<-.0ex>[r] \ar@<.7ex>[r]  & \mod_{\stinf(E)}(T_E^{\otimes 2}) \hdots}\right).
\end{equation*}
\end{thm*}

The crucial role played by quasi-elementary sub-Hopf algebras of $A$ is not a surprise in light of the properties recalled at the beginning of this section.
However, the formulation of this property given in Theorem \ref{thm:totalization} allows to tackle with the two following problems in a homotopy theoretic way.

\begin{prob*}[Determination of the Picard group \ref{prob:pic}]
What is the group of stably invertible $A$-modules with respect to the tensor product?
\end{prob*}
\begin{prob*}[The lifting problem \ref{prob:lift}]
Let $M_E \in \stinf(E)$ for all $E\in\E$. Does there exist an $A$-module $M$ such that $U_EM=M_E$ for all $E \in \E$?
\end{prob*}

Recall that there is a refinement of the Picard group of $A$ as a space, which we denote $\pic(A)$ here (see Definition \ref{de:picardspace}). The set of connected components of this space has a natural group structure, which makes it isomorphic to the Picard group of $\st(A)$. Its higher homotopy groups are identified in the following way.

\begin{pro*}[{Proposition \ref{pro:pipic}}]
There is the following identification of the homotopy groups of the Picard space:
\begin{enumerate}
\item $\pi_0(pic(A))$ is the Picard group of $\st(A)$, that is the group of invertible elements with respect to the tensor product,
\item $\pi_1(pic(A)) = Aut_{A}(\mathbb{1})$.
\item $\pi_i(pic(A)) = \pi_{i-1}(\hom_A(\un,\un))$, if $i \geq 2$.
\end{enumerate}
\end{pro*}

The groups $\pi_{i-1}(\hom_A(\un,\un))$ are interpreted in terms of extension groups in $\st(A)$. Therefore, these are computable.

Since a totalization is a limit, one applies the limit preserving functor $\pic$ to the result of Theorem \ref{thm:totalization}, one has

\begin{thm*}[{Theorem \ref{thm:picardss}}]
There is a weak equivalence of spaces
$$ \pic(A) \cong Tot\left(  \xymatrix{\pic(Mod_{\stinf(A)}(T))
  \ar@<-.5ex>[r] \ar@<.5ex>[r] & \pic(Mod_{\stinf(A)}(T\otimes T)) \ar@<-.7ex>[r] \ar@<-.0ex>[r] \ar@<.7ex>[r]
& \hdots}\right) .$$
\end{thm*}

The associated Bousfield-Kan spectral sequence computes the homotopy groups of the Picard space. To avoid complications, we state this spectral sequence, given in all generality in Corollary \ref{cor:sspicard}, in a particular case.

\begin{cor}[Corollary \ref{cor:sspicard}]
Suppose that $A$ has exactly one maximal quasi-elementary Hopf subalgebra $E$. Then, there is a spectral sequence 
$$E_1^{n,s} = \pi_s(\pic(Mod_{\stinf(E)(T^{\otimes n})})) \Rightarrow \pi_{s-n}(\pic(A))$$
computing the homotopy groups of the Picard space.

Its $r$th differential has bidegree $(r,r-1)$, and this spectral sequence converges for $s-n \geq 0$.
\end{cor}

This spectral sequence is an algebraic version of the spectral sequence considered first in \cite{MNN15}, and studied also in \cite{MS14}.

 This spectral sequence gives the Picard group of the stable category of $A$-modules from the Picard groups of the stable categories of $E$-modules, where $E$ runs through the collection of quasi-elementary Hopf subalgebras of $A$. This is our main computational tool for Picard groups.

We then turn to the second, and more difficult question of building a module knowing only its restrictions to quasi-elementary sub-Hopf algebras $E$ of $A$. When the Hopf algebra $A$ has only one quasi-elementary sub-Hopf algebra $E$, this is a lifting problem: given an $E$-module $M$, how many $A$-module whose restriction to $E$ is isomorphic to $M$ are there? It turns out that there is a homotopy theoretic refinement of this question, which is the space of lifts, whose $\pi_0$ is the set of classes of such lifts. Again, the nice $\infty$-categorical behaviour of this construction with respect to homotopy limits enables us to produce a machinery computing the homotopy groups of this space of lifts, by application of Theorem \ref{thm:totalization}.

\begin{cor*}[{Corollary \ref{thm:liftss}}]

Let $M_{\E}\in\stinf(\E)$. Let
\begin{equation}
E_1(M_{\E})^{s,n} = N^n(\pi_{s}\L_{\mod_{\stinf(A)}(T^{\bullet+1})}(M_{\E})).
\end{equation}
There are  obstruction classes to the lift problem in $E_1^{s,0,s+1}$ for each $s \geq 1$. In particular, $M_{\E}$ is liftable whenever these groups are zero.

Suppose that $M_{\E}$ is a liftable module, and let $M \in \stinf(A)$ be such a lift. There is a Bousfield-Kan spectral sequence
$$E_1(M_{\E})^{s,n} \Rightarrow \pi_{s-n}( \L_{\stinf(A)}(M_{\E}),M).$$

This spectral sequence is called the spectral sequence for the space of lifts.
\end{cor*}

Note that in this case, we obtain an obstruction theory for such lifting problem for free.

Finally, we study some multiplicative structures on these spectral sequences. Although we do not need these for our purposes, we believe that these are essential in more advanced computations.

We then set up a version of the Cartan Eilenberg spectral sequence 

$$E_1^{s,t,n}(\End) \Rightarrow \pi_{s,t}(\End_{\stinf(A)}(\un))$$
and an analogous spectral sequence 
$$E_1^{s,t,n}(\gm) \Rightarrow \pi_{s,t}(\gm_{\stinf(A)}(\un)$$
using the same techniques as $E_1(\pic)$.  There is a comparison between these spectral sequences that helps us understand the spectral sequence for the picard group from $E_1(\End)$ given by Proposition \ref{pro:comparison}. Moreover, we can give an extra-structure to the latter.

\begin{pro*}[{Proposition \ref{pro:pairinginss}}]
Let $\C$ be a pointed stable symmetric monoidal $\infty$-category  with distinguished object $M_{\C}$.
There are pairings
\begin{equation*}
\End_{\C}(\un) \wedge \End_{\C}(M_{\C}) \rightarrow \End_{\C}(M_{\C})
\end{equation*}
and 
\begin{equation*}
\gm_{\C}(\un) \wedge \gm_{\C}(M_{\C}) \rightarrow \gm_{\C}(M_{\C}).
\end{equation*}
These induces pairings of spectral sequences
\begin{equation*}
E_1^{s,t,n}(\End) \otimes E_1^{s',t',n'}(\End(M)) \rightarrow E_1^{s+s',t+t',n+n'}(\End(M))
\end{equation*}
and
\begin{equation*}
E_1^{s,n}(\gm) \otimes E_1^{s',n'}(\gm(M)) \rightarrow E_1^{s+s',n+n'}(\gm(M))
\end{equation*}
converging to the evident pairing in homotopy.
\end{pro*}

This gives a pairing which helps to understand the differentials for the spectral sequences for endomorphisms, which can in turn be compared to the spectral sequence for the space of lifts.

We end the paper by an application to the case when $A = \A(1)$, the sub-Hopf algebra of the modulo $2$ Steenrod algebra generated by the two first Steenrod squares $Sq^1$ and $Sq^2$.

We give a new computation of the well-known Picard group of the stable category of $\A(1)$-modules,

\begin{thm*}[{Theorem \ref{thm:pica1}}]
$$Pic(\A(1)) = \Z \oplus \Z \oplus \Z/2.$$
\end{thm*}

and we answer the lifting problem for two extreme cases in  $\E(1)$-modules. Although the inspection is tedious, the tools provided in this paper are sufficient to solve the lifting problem in this case, for any $\E(1)$-module.

This last question was already studied by Geoffrey Powell in \cite{Pow15}. In \textit{loc cit}, he is able to compute by hand the number of lifts of each indecomposable $\E(1)$-module.

The author is aware that Eric Wofsey has obtained results in a similar direction in an unpublished work. \\

The organization of the paper should now be clear from the introduction.

\tableofcontents

\part{Generalities about the stable $\infty$-category of modules}

\section{The stable $\infty$-category}

Let $A$ be a Hopf algebra. Classically, one builds the stable category $\st(A)$ of $A$-modules from the category $_A\mod$ by setting the maps that factors through a projective $A$-module to be zero. In particular, it forces the projective modules to be contractible in $\st(A)$.

In this section, we consider the stable category of $A$-modules as an $\infty$-category. We will see later that this category is actually a stable symmetric monoidal $\infty$-category. Although this structure does translate quite easily to elementary properties of the classical stable category of $A$-modules $\st(A)$, this observation is the starting point of our analysis. Indeed, we will profit from these $\infty$-categorical structures by using homotopy theoretical tools in representation theory of Hopf algebras.

We start by some recollection about graded modules over graded Hopf algebras.

\begin{pro}[\emph{\cite[Theorem 12.5, Proposition 12.8]{Mar83}}]
The category ${_A\mod}$ has enough injectives and projectives. Moreover, the notions of free, injective, projective and flat module coincide.
\end{pro}

We now fix some notations to take into account the fact that our objects are graded.

\begin{de}
For an $A$-module $M \in {_A\mod}$, we denote the $\F$-vector space of homogeneous elements of $M$ in degree $k$ by $M_k$.
There is a shift functor, that sends $M$ to $M[1]$, defined by $(M[1])_{k+1} = M_k$. This is clearly an invertible functor, and we denote by $(-)[t]$, for $t \in \Z$ the obvious $|t|$-fold composition.
\end{de}

The reader should be warned about this shift functor: it has nothing to do with the stable structure mentioned in the introduction. It is a simple regrading, and its importance in this paper is only for computational purposes, as the internal grading helps to organize data (for example extension groups are bigraded), and to obtain additional structure (the various rings that appear in our context usually come from graded rings).

The category of $A$-modules also comes with a closed symmetric monoidal structure, induced by the (cocommutative) coproduct on $A$. The interested reader is referred to \cite{Mar83} or \cite{Pal01} for the details of the construction.

\begin{de}
Let $M,N$ be two $A$-modules. Let $M \otimes N$ be the tensor product (over $\F$) of $M$ and $N$, together with the action of $A$ given by
\begin{equation}
a(m \otimes n) = \Sigma a'm \otimes a''n
\end{equation}
in Swedler's notations.
Let $\hom_A(-,-)$ be the internal $\hom$ functor which is part of this closed symmetric monoidal structure.
\end{de}

We now recall the construction of the stable category of $A$-modules as a symmetric monoidal model category.

\begin{de}
We define the following classes of morphisms in $_A\mod$. An morphism $f$ in the category of $A$-modules is
 \begin{enumerate}
\item a stable equivalence, if both its kernel and cokernel are free $A$-modules,
\item a cofibration, if it is an injection,
\item a fibration if it is a surjection.
\end{enumerate}
\end{de}

Note that the shearing isomorphism ensures that for any $A$-module $M$, the $A$-module $A \otimes M$ is free, so that the bifunctor $\otimes$ sends free modules to free modules. Consequently, $\otimes$ is compatible with the above classes of maps. This gives:

\begin{pro}[\emph{ \cite[Exemple 2.4.(v)]{SS03}}] \label{pro:stablecat}
The classes of stable equivalences, cofibrations, and fibrations define a symmetric monoidal closed model structure on the category of $A$-modules. 
\end{pro}

Moreover, there is a stability property for this model category, in the sense of Hovey-Palmieri-Strickland.

\begin{pro}[\emph{ \cite[Section 9.6]{HPS} }] \label{pro:hpsstable}
The category $\st(A)$ is a stable homotopy category in the sense of Hovey-Palmieri-Strickland. In particular, it is a triangulated closed symmetric monoidal category generated by the compact objects, $\un[t]$. \\
 
Precisely, \begin{enumerate}
\item the exact triangles are induced by short exact sequences of $A$-modules,
\item the desuspension functor  is   $\Omega := \ker(A \rightarrow \un) \otimes (-)$,
\item the symmetric closed monoidal structure is induced by  $\otimes$, the internal hom, and the unit is $\un$.
\end{enumerate}
\end{pro}

\begin{de}
We denote by $\stinf(A)$ the corresponding presentable $\infty$-category.
\end{de}

\begin{cor}[\emph{Corrolary of Propositions \ref{pro:stablecat} and \ref{pro:hpsstable}}]\label{de:loop}\label{lemma:comparisonofloops}
The $\infty$-category $\stinf(A)$ is a presentable symmetric monoidal stable $\infty$-category, generated by $\un[t]$, for $t \in \Z$. Moreover, the triangulated desuspension is induced by the functor $\Omega$.
\end{cor}

Note that the morphisms between two object $M$ and $N$, \textit{i.e.} the elements of $\pi_0(\stinf(A)(M,N))$ are the stable classes of \emph{degree zero} $A$-module morphisms between $M$ and $N$. However, we define an enhancement of $\stinf(A)$ which takes into account the maps of degree non-zero in the following way.

\begin{de}
Let $t \in \Z$. Let $\stinf(A)(M,N)_t$ be the spectrum $\stinf(A)(M, N[t])$, and $\stinf(A)(M,N)_*$ be the coproduct of all $\stinf(A)(M,N)_t$.
\end{de}

In particular, this gives a natural definition of bigraded homotopy groups in this category. This is nothing but a topological manifestation of the fact that extension groups between graded modules over a graded ring are bigraded, by one homological degree and one degree coming from the internal grading of objects.

\begin{de}
\begin{equation} \label{eq:bigradedpi}
\pi_{s,t}(M,N) := \pi_s(\stinf(A)(M,N)_t)
\end{equation}
\end{de}

\section{Picard spectra, endomorphisms, automorphisms and  space of lifts}

In this section, we define the various $\infty$-functors we will need to set up our spectral sequences for the Picard space and the space of lifts. Because our strategy relies on general $\infty$-categorical constructions, we need to consider these functors for a general stable $\infty$-category $\C$, which satisfies the relevant hypothesis. However, when particularized to the symmetric monoidal stable $\infty$-category $\stinf(A)$, we recover the objects that are our main interest here.

\begin{de} \label{de:picardspace}
For $\C$ a stable symmetric monoidal $\infty$-category, let $\pic(\C)$ denote the $\infty$-groupoid of invertible objects in $\C$, equivalences between them, and all higher arrows. As $\pic(\C)$ is a group like $E_{\infty}$-space, it is the zeroth space of a connective $\Omega$-spectrum we denote $\pic(\C)$ as well.
If $\C = \stinf(A)$ for some Hopf algebra $A$, we simply denote this spectrum by $\pic(A)$.
\end{de}

The Picard spaces and spectra are studied through their relation to some other fundamental functors. Indeed, since the higher morphisms in $\pic(\C)$ consists in all the invertible morphisms, each connected component of the Picard space is a delooping of the automorphisms of the unit. We now define properly the relevant functors and natural transformations.

\begin{de}
Let $\End_{(-)}(\un) : \infty\dashmod \cat^{\otimes} \rightarrow \sh$, where $\sh$ is the stable homotopy category, and $\infty\dashmod \cat^{\otimes}$ is the category of stable symmetric monoidal $\infty$-categories, be the functor which associates to a stable symmetric monoidal $\infty$-category its spectrum of endomorphisms of the unit, and $\gm_{-}(\un) = \iota \End_{(-)}(\un)$ its core.
\end{de}

The Picard spectrum satisfies the following property, which is essential to our approach. Indeed, it expresses in particular the compatibility of Picard spectra with totalizations.

\begin{pro}[Proposition 2.2.3 \cite{MS14}] \label{pro:piclimits}
The functor
$$\pic : \infty\dashmod cat^{\otimes} \rightarrow \infty\dashmod groupoids$$
preserves limits and filtered colimits.
\end{pro}

As the name suggests, the Picard spectrum is an enhancement of the more classical Picard group, defined as follows.

\begin{de}
Let $(C,\otimes, \un)$ be a monoidal category (such that for example the homotopy category of a stable symmetric monoidal $\infty$-category). The Picard group of $C$, denoted $\Pic(C)$, is the class of isomorphism classes invertible objects with respect to the tensor product. Whenever $\Pic(C)$ is a set, it has a natural group structure, where the multiplication is induced by $\otimes$ and the unit is the isomorphism class of the unit.
\end{de}

\begin{rk} \label{rk:picdeloopingaut}
By definition, there is a natural equivalence $\tau_{\geq 1}\pic(\C) \steq \Omega \gm_{\C}(\un)$.
The latter is a subspace of the space of endomorphisms of the unit, giving a natural transformation $\gm_{\C}(\un) \rightarrow \End_{\C}(\un)$.
\end{rk}

\begin{pro} \label{pro:pipic}
Let $\C$ be an symmetric monoidal $\infty$-category. There are natural isomorphisms
\begin{enumerate}
\item between $\pi_0(\pic(C))$ and the Picard group of $ho(\C)$,
\item $\pi_1(\pic(C)) \cong \gm_{\C}(\mathbb{1})$,
\item $\pi_i(\pic(C)) \cong \pi_{i-1}(\hom_C(\un,\un))$, for $i \geq 2$.
\end{enumerate}
\end{pro}

The homotopy groups of $\gm$ and $\End$ will be expressed as certain extension groups in Corollary \ref{cor:pdforend}, using a form of Poincar\'e duality  in the stable category of modules over a Hopf algebra. At this point, we will have an entirely algebraic description of the higher homotopy groups of $\pic(A)$.

We finish this section by the study of the homotopical behaviour of the various functors we have introduced. Our main interest is the commutation with homotopy limits. This is essential for us since our strategy is to apply $\pic$, $\gm$, $\End$ to a totalization of a cosimplicial $\infty$-category.

\begin{lemma} \label{lemma:holimend}
The functors $\End_{(-)}(\un)$ and $\gm$ commutes with homotopy limits. In particular, these commutes with totalization.
\end{lemma}

\begin{proof}
The functor $\End_{(-)}(\un)$ can be written as a composite
\begin{equation*}
\infty\dashmod \cat^{\otimes} \overset{F}{\rightarrow}  \infty\dashmod\cat \rightarrow \Top
\end{equation*}
where $F$ takes symmetric monoidal $\infty$-category and sends it to its full sub-$\infty$-category with one object, $\un$. The second functor associates to an $\infty$-category the space of all its maps. These functors both commute with homotopy limits, so $\End$ does.

Now, the functor $\gm$ commutes with limits as the composite of the two functors $\End_{(-)}(\un)$ and taking only invertible connected components, both commuting with limits. 
\end{proof}

Lemma \ref{lemma:holimend} does give a particular role to the unit in any symmetric monoidal $\infty$-category. If our categories come with a distinguished object, we can Apply a similar argument.

\begin{lemma} \label{lemma:endMandlimits}
Let $\infty\dashmod \cat^{\bullet}$ be the category of pointed stable symmetric monoidal $\infty$-category, \textit{i.e.} categories $\C$ with a distinguished object $M_{\C}$. The functors
\begin{equation*}
\End : \infty\dashmod \cat^{\bullet} \rightarrow \sh,
\end{equation*}
and
\begin{equation*}
\gm : \infty\dashmod \cat^{\bullet} \rightarrow \sh,
\end{equation*}
which sends $\C$ to $\End_{\C}(M_{\C},M_{\C})$ and $\gm_{\C}(M_{\C})$ commutes with limits.
\end{lemma}

\begin{proof}
The proof is along the same lines as Lemma \ref{lemma:holimend}, replacing the functor $F$ by the functor which sends a pointed stable symmetric monoidal $\infty$-category $\C$ to its full subcategory with only one object: $M_{\C}$.
\end{proof}

\section{Restriction functors}

In this section, we study and reformulate various restriction functors in the stable setting, defined as follows.

\begin{de} \label{de:adjforget}
Let $B$ be a Hopf subalgebra of $A$. Define
$$U_{B} : \stinf(A) \rightarrow \stinf(B)$$
to be the restriction functor.
\end{de}

 Conceptually, the central result of this section is Proposition \ref{pro:extofscalarshopf}, which reformulates $U_B$ as an extension of scalars along the commutative algebra object $T_B \in \stinf(A)$. This allows us to study the descent problem along $T_B$.

The following result is classical, see for instance \cite{Mar83,Pal01}.

\begin{pro}
The functor $U_B$ has both a left and a right adjoint. Namely,

$$A \otimes_B (-) : \stinf(B) \rightarrow \stinf(A)$$
is left adjoint to $U_B$, and
$$\hom_B(A ,-) : \stinf(B) \rightarrow \stinf(A)$$
is right adjoint to $U_B$,
where the action of $A$ on $A \otimes_B (-)$ and $\hom_B(A ,-)$ is induced by the action of $A$ on itself.
\end{pro}

\begin{pro} \label{pro:ubmonoidal}
The functor $U_B$ is a symmetric monoidal functor. 
Moreover, the three functors $U_{B}, A\otimes_{B}$, and $\hom_B(A,-)$ are exact.
\end{pro}

\begin{proof}
As vector spaces, $U_B(M \otimes N)$ is clearely isomorphic to $U_B M \otimes U_B N$, since the tensor product is taken over $\F$. The action of $B$ coincides since $B$ is a Hopf subalgebra of $A$.

The second assertion is immediate since tensor product is taken over $\F$.
\end{proof}

\begin{rk}
The functors $ A\otimes_{B}$, and $F_B(A,-)$ are obviously not monoidal in general, for instance, they do not preserve the units.
\end{rk}

It will often be useful in our situation to consider a family of Hopf subalgebras of $A$ at once. This is done by the following global restriction functor.

\begin{de} \label{de:restrictions}
Let $\B$ be a collection of Hopf subalgebras of $A$.
Let
\begin{equation*}
U_{\B} : \stinf(A) \rightarrow \stinf(\B),
\end{equation*}
where $\stinf(\B) := \coprod_{B\in\B} \stinf(B)$, be the functor $\coprod_{B\in\B} U_B$.
\end{de}

Let $B$ be a Hopf subalgebra of $A$. In order to study the descent from $B$ to $A$ modules, we need to have a commutative ring object in $\stinf(A)$ that measures the difference between $\stinf(A)$ and $\stinf(B)$. 

The Hopf algebra $A$ is cocommutative. It implies that all its Hopf subalgebras are automatically conormal subcoalgebras. Consequently, the Hopf subalgebra $B$ of $A$ is conormal, and $A\sur B$ is a cocommutative coalgebra object in $\stinf(A)$.

\begin{de} \label{de:defT}
Let $B \subset A$ be a Hopf subalgebra.
Then  $A\sur B = A \otimes_B \F$ and the diagonal of $A$ induces a cocommutative coalgebra structure on it. In particular, the $\F$-linear dual $(A\sur B)^{\ast}$ of $A\sur B$ is a commutative algebra object in the category $\stinf(A)$.
We denote this commutative algebra by $T_B$.
\end{de}

\begin{lemma} \label{lemma:hombaconservative}
The functor $\hom_B(A,-) : \stinf(B) \rightarrow \stinf(A)$ is conservative.
\end{lemma}

\begin{proof}
We argue at the level of the module categories for this result. Recall that a map in the stable category is a weak equivalence if and only if both its kernel and cokernel are free. Thus, it suffices to show that for any $B$-module $M$,  if $\hom_B(A,M)$ is a free $A$-module then $M$ is a free $B$-module. It is true since the morphism $M \rightarrow U_B(\hom_B(A,M)$ is split, and $U_B(\hom_B(A,M)$ is a free $B$-module.
\end{proof}

The last result is the key ingredient in the following proposition, which exhibits $B$-modules in $A$ as a category of modules over the commutative algebra object $T_B \in \stinf(A)$.

\begin{pro} \label{pro:extofscalarshopf}
The functors $U_{B}$ induces an equivalence of stable symmetric monoidal $\infty$-categories 
$$ \mod_{A}(T_B) \cong \stinf(B),$$
where $\mod_{A}(T_B)$ is the category of $T_B$-modules in $\stinf(A)$. Moreover, this equivalence of categories enters in a diagram
\begin{equation}
\xymatrix{ \stinf(A) \ar[r]^{U_B} \ar[d]_{T_B \otimes}  & \stinf(B) \ar@{-}[dl]^{\cong} \\
\mod_A(T_B), }
\end{equation}

commutative up to a natural isomorphism.
\end{pro}

\begin{proof}
We check the hypothesis of Proposition 5.29 \cite{MNN15} for the adjunction
$$ U_B : \stinf(A) \rightleftarrows \stinf(B) : \hom_{\stinf(B)}(A,-)$$
of Definition \ref{de:adjforget}. Namely, we need to show that\begin{itemize}
\item the functor $U_B$ is monoidal,
\item for all objects $M \in \stinf(A)$, $N \in \stinf(B)$, the map
$$ \hom_B(A,N) \otimes M \rightarrow \hom_B(A, (N \otimes U_B(M))$$ is a stable equivalence,
\item the functor $\hom_B(A,-)$ commutes with colimits,
\item the functor $\hom_B(A,-)$ is conservative.
\end{itemize}

The fact that $U_B$ is monoidal is given by Proposition \ref{pro:ubmonoidal}. \\

The algebras $A$ and $B$ are finite dimensional. In particular $A$ is a compact $B$-module. Thus, for all $M \in \stinf(A)$ and $N \in \stinf(B)$, the morphism
$$\hom_B(A,N) \otimes M \rightarrow \hom_B(A, N \otimes M)$$
is a stable equivalence, and $\hom_B(A,-)$ commutes with colimits.

The projection formula holds since the category $\stinf(B)$ is generated by the strongly dualizable objects $\un[k]$.  \\

Finally, $\hom_B(A,-)$ is conservative by Lemma \ref{lemma:hombaconservative}.
\end{proof}

The last proposition is one key point for the study of the restriction functors via descent theory. However, in our arguments, we will often need to consider a restriction functor of the form $U_{\B} : \stinf(A) \rightarrow \coprod_{\B} \stinf(B)$, as in Definition \ref{de:restrictions}. To this end, we use the following consequence of Proposition \ref{pro:extofscalarshopf}.

\begin{cor} \label{cor:assemblyextension}
Let $T_{\B}$ be the algebra $\prod_{B \in \B} T_B$. Then there is an equivalence of stable monoidal $\infty$-categories 
$$ \mod_{\stinf(A)}(T_{\B}) \cong \stinf(\B),$$
where $\mod_{\stinf(A)}(T_{\B})$ is the category of $T_{\B}$-modules in $\stinf(A)$. Moreover, this equivalence of categories enters in a diagram
\begin{equation}
\xymatrix{ \stinf(A) \ar[r]^{U_{\B}} \ar[d]_{T_{\B} \otimes}  & \stinf({\B}) \ar@{-}[dl]^{\cong} \\
\mod(T_{\B}), }
\end{equation}

commutative up to a natural isomorphism.

\end{cor}

\section{Homotopy groups, extension groups and Poincar\'e duality} \label{sec:poincare}

This section is devoted to the algebraic interpretation of the homotopy groups of the morphism spaces. As the reader probably expects, and as the expert probably knows, these are related to extension groups in the stable category.

Here, we recollect the two essential ingredients needed to prove this fact. Namely the identification of $\pi_{s,t}(\stinf(A)(M,N))$ with $\ext_A^{-s,t}(M,N)$, which is given by Proposition \ref{pro:piiextgroups}, and the Poincar\'e duality result for such extension groups, provided by Proposition \ref{pro:poincare}.

Recall that, in the stable category, the extension groups $\ext^{s,t}_A$ are defined for all integers $s,t$.

\begin{pro} \label{pro:piiextgroups}
There is a natural isomorphism
\begin{equation*}
\pi_{s,t}(\stinf(A)(M, N)) \cong \ext_{A}^{-s,t}(M,N).
\end{equation*}
\end{pro}

\begin{proof}
Clearly, $\pi_{0,t}(\stinf(A)(M,N)) = \st(A)(M,N)_t$, the set of all stable module morphisms between $M$ and $N$. But  we know by Lemma \ref{lemma:comparisonofloops} that the functor $\Omega$ coincides with the desuspension in the stable symmetric monoidal $\infty$-category $\stinf(A)$. Therefore 
\begin{eqnarray*}
\pi_{s,t}(\stinf(A)(M,N)) & =  & \pi_{0,t}(\stinf(A)(\Omega^s M,N)) \\
  & \cong & \st(A)(\Omega^s M,N)_t  \\
& \cong & \ext^{-s,t}_A(M,N).
\end{eqnarray*}
\end{proof}

Unfortunately, the computations that are usually available for extension groups are done in the unstable category (\textit{i.e.} $\ext^{s,t}$ for $s \geq 0$). To remedy to this problem, we show a Poincar\'e duality result. In particular, the computation of the $s \geq 0$ part and the $s \leq -1$ part of extension groups are essentially equivalent. This sort of computations is classical, see for example \cite[Chapter 12]{Mar83}.

\begin{de}
Let $|A|$ be the biggest degree in which $A$ is non-zero.
\end{de}

In particular, the non-degenerate pairing $A \otimes A \rightarrow A \rightarrow \un[|A|]$, where the first map is the multiplication in $A$, and the second map projects on the top class, induces an isomorphism of $A$-modules $A^{\ast} \cong A[-|A|]$, where $A^{\ast}$ is the $\F$-linear dual of $A$.

\begin{rk} \label{rk:dualizeresolutions}
As a consequence of this isomorphism is that one can obtain a projective resolution of an $A$-module $M$ by dualizing an injective resolution for $M^{\ast}$, and vice versa. This gives the following result.
\end{rk}

\begin{pro} \label{pro:poincare}
There is a natural isomorphism
\begin{equation*}
\ext_A^{s,t}(\un,\un) \cong  \ext_A^{-1-s,|A|-t}(\F,\F)  .
\end{equation*}
\end{pro}

\begin{proof}
Choose a projective resolution of $\un$, and dualize it to obtain an injective resolution of $\un$ (use Remark \ref{rk:dualizeresolutions}). Then, the isomorphism $A^{\ast} \cong A[-|A|]$  gives the statement.
\end{proof}

In particular, we get the desired identification of homotopy groups of $hom$-spectra in $\stinf(A)$.

\begin{cor} \label{cor:pdforend}
There is a natural isomorphism
 
\begin{equation*}
\pi_{s,t}(\hom_A(\un,\un)) \cong \ext_A^{s-1,|A|-t}(\un,\un)^{\ast}.
\end{equation*}
\end{cor}

We conclude this section by various identifications of homotopy groups of spaces ($\pic$, $\End$, $\gm$) which are relevant to our study.

\begin{cor}
There are identifications
 $\pi_i(\gm_{\stinf(A)}(\un)) \eq \ext_A^{i-1,|A|}(\un, \un)^{\ast}$, for $i \geq 1$, and $\pi_0(\gm_{\stinf(A)}(\un) = \{ \mathrm{id}_{\un} \}$.
\end{cor}

Using Corollary \ref{cor:pdforend}, we can also identify the higher homotopy groups of $\pic(A)$.

\begin{pro} \label{pro:picpiiviaduality}
\begin{enumerate}
\item $\pi_1(\pic(A)) = \{ \mathrm{id}_{\un} \}$,
\item for $i \geq 2$, $\pi_i(\pic(A)) = \ext^{i-2,|A|}(\un,\un)^{\ast}$.
\end{enumerate}
\end{pro}

\section{Quillen's F-isomorphism theorem for Hopf algebras}

In this section, we recall and reformulate in our context the classical Quillen's F-isomorphism theorem for detecting nilpotent elements in the cohomology of Hopf algebras. Quillen's original result is proved in \cite{Qui71}, and the corresponding result for cohomology of Hopf algebras, that is $\ext^*_A(\F,\F)$ for a Hopf algebra $A$, can be found in \cite[Theorem 1.3]{Pal97}.

Let $G$ be a finite group, for simplicity. Consider the collection $\E$ of elementary abelian $2$-subgroups of $G$. Quillen's F-isomorphism theorem asserts that the map
\begin{equation*}
H^*(G;\F) \rightarrow \lim_{E \in \E } H^*(E;\F)
\end{equation*}
is an F-isomorphism (finite kernel and each element in the target has some power in the image).

In the theory of Hopf algebras, the appropriate version of elementary abelian $2$-subgroup is less tractable. It is the notion of quasi-elementary Hopf subalgebra.

\begin{de}[\emph{ \cite[Definition 2.1.10]{Pal97}}] \label{de:quasielem}
We say that a connected cocommutative Hopf algebra $B$ is quasi-elementary if no product of the form
$$\prod_{\omega \in S} \omega$$
is zero, for $S$ a finite subset of $\ext^{1,*}_B(\F,\F)$.
\end{de}

Using this notion of quasi-elementary Hopf subalgebra of $A$, the appropriate version of Quillen F-isomorphism theorem holds (see \cite[Theorem 1.2]{Pal97}). The results contained in this section can be seen as the direct consequences of \cite[Theorem 1.2]{Pal97} in representation theory of Hopf algebras.

The hypothesis of Definition \ref{de:quasielem} are certainly satisfied when $B$ is an exterior Hopf algebra, since $\ext^{*,*}_B(\F,\F)$ is polynomial in this case. This motivates the following definition.

\begin{de}
A  Hopf algebra $E$ is called elementary if it is commutative, and $x^2=0$ for all $x$ in the augmentation ideal of $E$.
\end{de}

\begin{lemma}
Elementary  Hopf algebras  are quasi-elementary.
\end{lemma}

\begin{proof}
Follows from an easy computation of the cohomology of exterior Hopf algebras.
\end{proof}

For many Hopf algebras, these two notions actually coincide, for instance, this is the case for all sub-Hopf-algebras of the modulo $2$ Steenrod algebra, by \cite[Section 2.1.1]{Pal01}.

\begin{de} \label{de:mathcale}
Let $\E$ be the set of maximal quasi-elementary sub-Hopf-algebras of $A$.
\end{de}

We now turn to a central object of our study. By Definition \ref{de:restrictions} and Corollary \ref{cor:assemblyextension}, the forgetful functors $U_E$, for $E \in \E$ assemble into a functor 
\begin{equation*}
U_{\E} : \stinf(A) \rightarrow \stinf(\E),
\end{equation*}
which is isomorphic to
\begin{equation*}
T_{\E} \otimes  : \stinf(A) \rightarrow \mod(T_{\E}),
\end{equation*}
through the equivalence $\stinf(\E) \cong \mod(T_{\E})$. The next result gives one of the essential features of the functor $U_{\E}$ (or equivalently, of the algebra $T_{\E}$).

\begin{pro} \label{pro:detection}
An $A$-module $M$ is free if and only if  $U_{E}(M)$ is $E$-free, for all $E \in \E$. Consequently, the functor $U_{\E}$ is conservative.
Likewise, an $A$-module $M$ is in $Pic(A)$ if and only if, for all $E \in \E$, the stable $E$-module $U_{E}(M)$ is in $Pic(E)$.
\end{pro}

\begin{proof}
The first assertion is \cite[Theorem 1.3]{Pal97}. Now, for $B$ a Hopf algebra, a map $f : M \rightarrow N$ between $B$-modules is a stable isomorphism if and only if its cofiber is free.
Thus a morphism $f$ in $\st(A)$ is a stable isomorphism if and only if, for every $E \in \E$, $U_{E}f$ is a stable isomorphism of $E$-modules. \\

We still have to show the last property. Observe that $M$ is invertible with respect to the tensor product if and only if the map $$M \otimes \hom(M,\un) \rightarrow \un$$
is a stable equivalence. The fact that each $U_{E}$ is monoidal (see Remark \ref{pro:ubmonoidal}), and that stable equivalences are detected on $\E$ gives the result.
\end{proof}

\part{Descent, Picard groups, and the lifting problem} \label{part:descent}

\section{A cosimplicial model for $\stinf(A)$}

In this section, we show how the stable category $\stinf(A)$ can be recovered from the categories $\stinf(E)$, where $E$ runs through the collection of maximal elementary Hopf subalgebras of $A$.

For simplicity, we denote simply by $T$ the algebra object $T_{\E}$, where $\E$ is the set of all maximal elementary Hopf subalgebras of $A$ (as in Definition \ref{de:mathcale}). The main result of this section is Theorem \ref{thm:totalization} which is a homotopical descent result for the algebra object $T$.

Following the analogy between representations of a group $G$ and the study of representations of a connective Hopf algebra $A$, one can think of this result as an analogue of the descent up-to-nilpotence set-up in \cite{MNN15}. Indeed, each of the algebras $T_E$, for $E\in \E$, is the correct analogue of a $G$-orbit in our setting.

\begin{thm} \label{thm:totalization}
The functor $T \otimes : \stinf(A) \rightarrow \mod_{\stinf(A)}(T)$
extends to an equivalence of stable monoidal $\infty$-categories
\begin{equation}
 \stinf(A) \cong Tot\left(  \xymatrix{\mod_{\stinf(A)}(T)
  \ar@<-.5ex>[r] \ar@<.5ex>[r] & \mod_{\stinf(A)}(T\otimes T) \ar@<-.7ex>[r] \ar@<-.0ex>[r] \ar@<.7ex>[r]
& \hdots}\right),
\end{equation}
where $T$ is the algebra $ \Pi_{E \in \E} T_E.$
When $A$ admits only one elementary sub-Hopf-algebra $E$, this reduces to
\begin{equation*} 
 \stinf(A) \cong Tot\left(  \xymatrix{\stinf(E) \ar@<-.5ex>[r] \ar@<.5ex>[r] & \mod_{\stinf(E)}(T_E) \ar@<-.7ex>[r] \ar@<-.0ex>[r] \ar@<.7ex>[r]  & \mod_{\stinf(E)}(T_E^{\otimes 2}) \hdots}\right).
\end{equation*}
\end{thm}

In other words, the stable category of $A$-modules can be reconstructed from the stable category of $E$-modules, for each $E \in \E$. Essentially, this is due to the fact that the functor $U_{\E}$, or equivalently $T \otimes : \stinf(A) \rightarrow \stinf(A)$, does not loose too much information. This property, which is also the crucial point of the proof of Theorem \ref{thm:totalization}, is expressed in the following Lemma.

\begin{lemma} \label{lemma:tconservative}
The functor 
$$T \otimes(-) : \stinf(A) \rightarrow \stinf(A)$$
is conservative.
\end{lemma}

\begin{proof}
By Proposition \ref{pro:detection}, the functor $U_{\E}$ is a conservative functor. But by Corollary \ref{cor:assemblyextension}, the functor 
$$T \otimes(-) : \stinf(A) \rightarrow \mod_{\stinf(A)}(T)$$
is the composition of $U_{\E}$ and an equivalence of categories. Thus, 
it is conservative. In particular, the functor 
$$T \otimes(-) : \stinf(A) \rightarrow \stinf(A)$$
is conservative as well.
\end{proof}

\begin{proof}[Proof of Theorem \ref{thm:totalization}]
We show that the algebra object $T$ satisfies the hypothesis of Theorem 2.30 \cite{MNN15}, which identifies the category of $T$-complete objects with the desired totalization. 

Tensoring with the algebra $T$ is conservative by Lemma \ref{lemma:tconservative}. Moreover, $T$ is strongly dualizable, since it is a finite dimensional $A$-module.

Thus, \cite[Theorem 2.30]{MNN15} provides an equivalence of stable monoidal $\infty$-categories \begin{equation} \label{eq:crudetotalization}
 \stinf(A) \cong Tot\left(  \xymatrix{\mod_{\stinf(A)}(T)
  \ar@<-.5ex>[r] \ar@<.5ex>[r] & \mod_{\stinf(A)}(T\otimes T) \ar@<-.7ex>[r] \ar@<-.0ex>[r] \ar@<.7ex>[r]
& \hdots}\right).
\end{equation}
Now, by the identification given in Proposition \ref{pro:extofscalarshopf}, the layers $\mod_{\stinf(A)}(T^{\otimes n})$ of the cosimplicial category given in equation \eqref{eq:crudetotalization} are equivalent to $\mod_{\stinf(\E)}(T^{\otimes (n-1)})$.

To conclude the proof, we have to show that every object of $\stinf(A)$ is $T$-complete. Now, the general argument given in Example 2.17 \cite{MNN15} applies, since $T$ is finite-dimensional and tensoring with $T$ is a conservative functor by Lemma \ref{lemma:tconservative}.
\end{proof}

\begin{rk}
Theorem \ref{thm:totalization} is the key point of our analysis. It is tempting to change the definition of $\E$ to ease the computation. The analogous Theorem holds, replacing the set $\E$ by any set $\B$ consisting in Hopf subalgebras of $A$, satisfying the following property. For any maximal elementary Hopf subalgebra $E$ of $A$, there is a Hopf subalgebra $B \in \B$ such that $E \subset B$.

The reader must be warned that this does not seem to ease the computations in concrete examples.
\end{rk}

\section{The stable Picard group via descent} \label{sec:picarddescent}

In this section, we use Theorem \ref{thm:totalization} to address the following question.

\begin{prob}[Determination of the Picard group] \label{prob:pic}
What is the group of stably invertible $A$-modules with respect to the tensor product?
\end{prob}

As stated, it is not clear how the technology developed in the previous sections will help us in this study. In order to apply our results, we first need to consider an object whose behaviour with respect to homotopical constructions is nicer than the Picard group. Recall from Definition \ref{de:picardspace} that there is a connective spectrum $\pic(A)$, associated to the Hopf algebra $A$, which is a spectral enhancement of the Picard group in the sense of Proposition \ref{pro:picpiiviaduality}.

The reason why this object behaves better with respect to homotopical constructions than the mere Picard group is expressed in Proposition \ref{pro:piclimits}. In particular, we obtain

\begin{thm} \label{thm:picardss}
There is a weak equivalence of spaces
$$ \pic(A) \cong Tot\left(  \xymatrix{\pic(Mod_{\stinf(A)}(T))
  \ar@<-.5ex>[r] \ar@<.5ex>[r] & \pic(Mod_{\stinf(A)}(T\otimes T)) \ar@<-.7ex>[r] \ar@<-.0ex>[r] \ar@<.7ex>[r]
& \hdots}\right) .$$
\end{thm}

\begin{proof}
The first assertion is Theorem \ref{thm:totalization}, and the behaviour of $\pic$ with respect to limits given in Proposition \ref{pro:piclimits}. 
\end{proof}

\begin{cor} \label{cor:sspicard}
There is a spectral sequence 
$$E_1^{n,s} = N^n\pi_s(\pic(Mod_{\stinf(A)}(T^{\otimes {\bullet+1}}))) \Rightarrow \pi_{s-n}(\pic(A))$$
computing the homotopy groups of the Picard spectrum, where $N^n$ is the Moore normalization functor.

Its $r$th differential has bidegree $(r,r-1)$, and this spectral sequence converges for $s-n \geq 0$.
\end{cor}

\begin{proof}
The spectral sequence is the Bousfield-Kan spectral sequence associated to a cosimplicial space. Note that, because of Theorem \ref{thm:picardss}, we already know that the totalization of the cosimplicial space is non empty. Indeed, the unit $\un \in \stinf(A)$ is a canonical choice of base point in each layer, and it is also an element of the Picard group of $\stinf(A)$. Thus, we can run the spectral sequence.
\end{proof}

The last two results gives a general machinery which does the computation of $\pi_*(\pic(A))$. At this point, the computations are essentially inaccessible from such a generic description of the $E_1$-term of the spectral sequence for the Picard space. The rest of this section is devoted to the study of this spectral sequence. First, the Picard space (see Definition \ref{de:picardspace}) of a symmetric monoidal stable category $\C$ is a delooping of the space of automorphisms of the unit $\gm_{\C}(\un)$ (see Remark \ref{rk:picdeloopingaut}). But the latter is a subfunctor of $\End_{\C}(\un)$. We can thus compare the $E_1$-pages of the Bousfield-Kan spectral sequences derived from the cosimplicial spaces obtained by applying $\pic$, $\gm$ and $\End$ to the cosimplicial space appearing in Theorem \ref{thm:totalization}.

Furthermore, using the adjunction 
\begin{equation*}
T^{\otimes n} \otimes : \stinf(A) \rightleftarrows  \mod_{\stinf(A)}(T^{\otimes n}) : U,
\end{equation*} 
where $U$ denotes the forgetful functor, we obtain an explicit description of the latter spectral sequence.

We start by analysing the Bousfield-Kan spectral sequence for endomorphism spaces.

\begin{pro} \label{pro:bkssendaut}
There are two natural Bousfield-Kan spectral sequences:
\begin{equation}
E_1^{s,t,n}(\End) \cong N^n\ext_A^{s-1,|A|-t}(\un,T^{\otimes {\bullet+1}}) \Rightarrow \ext_A^{s+n-1,|A|-t}(\un,\un),
\end{equation}

and 
$$E_1^{s,n}(\gm) =  N^n\pi_s(\gm_{Mod_{\stinf(A)}(T^{\otimes \bullet+1})}(\un)) \Rightarrow \pi_{s-n}(\gm_{\stinf(A)}(\un)).$$
Moreover, the natural transformation $\gm_{-}(\un) \rightarrow \End_{(-)}(\un)$ induces a natural morphism of spectral sequences 
$$ E_1^{s,n}(\gm) \rightarrow E_1^{s,0,n}(\End)$$
which is an isomorphism in degrees $(s,n)$ for  $s \geq 1$ at the $E_1$-page.
\end{pro}

\begin{proof}
We start by analysing the former spectral sequence.
Apply the functor $\End_{(-)}(\un)$, which associates to an $\infty$-category its space of endomorphisms of the unit, to the equivalence of stable symmetric monoidal $\infty$-categories provided by Theorem \ref{thm:totalization}. This functor commutes with limits by Lemma \ref{lemma:holimend}, so we get a weak equivalence

$$ \End_{A}(\un) \cong Tot\left(  \xymatrix{\End_{Mod_{\stinf(A)}(T)}(\un)
  \ar@<-.5ex>[r] \ar@<.5ex>[r] & \End_{Mod_{\stinf(A)}(T\otimes T)}(\un) \ar@<-.7ex>[r] \ar@<-.0ex>[r] \ar@<.7ex>[r]
& \hdots}\right) .$$

There is an associated Bousfield-Kan spectral sequence, whose $E_1$-page is 
\begin{equation*}
E_1^{s,n}(\End) = \pi_s(\End_{(Mod_{\stinf(A)}(T^{\otimes {n+1}})}(\un)) \Rightarrow \pi_{s-n}(\End_{A}(\un) ).
\end{equation*}

Now, the layers of the cosimplicial space are identified with a more computable object, using the following chain of isomorphisms
\begin{eqnarray*}
\End_{(Mod_{\stinf(A)}(T^{\otimes {n+1}})}(\un) & := & \hom_{Mod_{\stinf(A)}(T^{\otimes {n+1}})}(T^{\otimes {n+1}},T^{\otimes n+1}) \\
& \simeq & \hom_{\stinf(A)}(\un, T^{\otimes {n+1}}).
\end{eqnarray*}

A similar argument is valid for the automorphism space functor. Apply the functor $\gm_{(-)}(-)$, which commutes with limits by Lemma \ref{lemma:holimend}, to the equivalence provided by Theorem \ref{thm:totalization}. The associated Bousfield-Kan spectral sequence is precisely the asserted one.

The identification of the map of spectral sequence given in the last assertion is induced by the natural transformation $\End_{(-)}(\un) \rightarrow \gm_{(-)}(-)$.
\end{proof}

Note that the unusual indexing of the stable extension groups is harmless: one could consider the spectral sequence 
\begin{equation*}
E'(\End)_1^{s,t,n} = N^n\ext_A^{s,t}(\un,T^{\otimes \bullet+1}) \Rightarrow \ext_A^{s-n,t}(\un,\un),
\end{equation*}
obtained by reindexing the one given in Proposition \ref{pro:bkssendaut}. The indexing is chosen here to have the comparison map 
\begin{equation*}
E_1^{s,n}(\gm) \rightarrow E_1^{s,0,n}(\End).
\end{equation*}

\begin{rk} \label{rk:simplicifationoneelementary}
In the case when there is only one maximal elementary sub-Hopf-algebra $E$ of $A$, the spectral sequence computing $\ext^{*,*}_A(\un,\un)$ given in Proposition \ref{pro:bkssendaut} coincides with a version the Cartan-Eilenberg spectral sequence from the $E_2$-page onwards. In this sense, the spectral sequence we obtain is a generalization of the latter.
\end{rk}

\begin{pro} \label{pro:comparison}
The canonical identification $\pi_{i-1}(\End_{(-)}(\un)) \eq \pi_i(\pic(-))$, for $i \geq 2$, induces an isomorphism  of abelian groups
\begin{equation}
E_1^{s-1,0,n}(\End) \cong E_1^{s,n}(\pic)
\end{equation}
for all $n \geq 0$ and $s \geq 2$.

Likewise, the canonical identification $\pi_{i-1}(\gm_{(-)}(\un)) \eq \pi_i(\pic(-))$, for $i \geq 1$, induces an isomorphism  of abelian groups
\begin{equation}
E_1^{s-1,n}(\gm) \cong E_1^{s,n}(\pic)
\end{equation}
whenever $n \geq 0$ and $s\geq 1$.

Moreover, for $i \geq 1$, the differentials originating at $E_r^{s,n}$ for $s-n \geq 1$ or $r \leq s-1$ in the descent spectral sequence converging to $\pi_{s-n}(\pic(A))$ are identified with the ones of $E^{s-1,0,n}_r(\End)$ (respectively $E^{s-1,n}_r(\gm)$).
\end{pro}

\begin{proof}
This is given in \cite[Section 5]{MS14}.
\end{proof}

Now, when given a Hopf algebra $A$, the strategy to compute the Picard group $\pi_0(\pic(A))$ goes as follows. First, we compute $\stext_A(\un,\un)$ and the $E_1$-page $E_1^{s,t,n}(\End)$. This gives some knowledge on the $E_{\infty}$-page of the latter spectral sequence. Then, we use the comparison result provided in Proposition \ref{pro:comparison}, and hope that is gives all the differentials in the spectral sequence for Picard groups given in Corollary \ref{cor:sspicard}. It turns out that in concrete situations, it is the case. We figure out some instances of this in the last part of this paper (see Part \ref{part:applications}).

\section{Lifting problem via descent} \label{sec:descentlift}

We now turn to a more general problem. Because of Palmieri's detection result (see Proposition \ref{pro:detection}), the computation of the stable Picard group of $A$ is strongly related to the determination of the set of all stable $A$-modules, whose restriction to all elementary Hopf subalgebras $E$ of $A$ are stably equivalent to $\un_{E}$. 

In this section, we deal with the following problem.

\begin{prob}[The lifting problem] \label{prob:lift}
Let $M_E \in \stinf(E)$ for all $E\in\E$. Does there exist an $A$-module $M$ such that $U_EM=M_E$ for all $E \in \E$?
\end{prob}

The reader who is familiar with this problem is probably disturbed by the absence of a compatibility condition on the various restrictions. Namely, one would expect that this problem is easier when one assumes that, for all $E' \subset E$, $U_{E'} M_E \simeq M_{E'}$. Unfortunately, the author does not know an analogue to Corollary \ref{cor:assemblyextension} in this case. This makes the computational part of our answer harder to set up.

Our approach to this problem is analogous to the one we used for Picard groups in Section \ref{sec:picarddescent}. We fix $\E$ and the category $\stinf(\E)$ in this section. Let $M_{\E}$ be a collection of stable representations $M_E$ of $E$, for each $E \in \E$.

We consider a category $\C$ which lives over $\stinf(\E)$, so that it makes sense to formulate the lifting problem in $\C$.

We first build a moduli space $\L_{\C}(M_{\E})$ for Problem \ref{prob:lift} in $\C$. This is Definition \ref{de:lifts}. Then, we show that this is in some sense functorial in $\C$, and that this functor preserves homotopy limits in $\C$. Finally, we obtain a spectral sequence computing the homotopy groups of $\L_{\stinf(A)}(M_{\E})$ as the Bousfield-Kan spectral sequence

\begin{de} \label{de:lifts}
Consider the category of stable symmetric monoidal $\infty$-categories over $\stinf(\E)$, {\it~i.e.} of diagrams
$$ \xymatrix{ \C \ar[d] \\ \stinf(\E)}.$$ The $\infty$-category of lifts of $M_{\E} \in \stinf(\E)$ (in $\C \rightarrow \stinf(\E)$) is the pullback
\begin{equation} \label{eq:spaceoflifts}
\xymatrix{ \tilde{\L}_{\C}(M_{\E}) \ar[r] \ar[d] & \C \ar[d] \\
[M_{\E}] \ar[r] & \stinf(\E) }
\end{equation}
where $[M_{\E}]$ is a shorthand for the full $\infty$-subcategory of $\stinf(\E)$ whose object is $M_{\E}$.
The space of lifts is $\L_{\C}(M_{\E}) := \iota \tilde{\L}_{\C}(M_{\E})$, the core of the $\infty$-category of lifts.
\end{de}

We first study a couple of examples of such spaces of lifts.

\begin{ex}
Consider $\stinf(\E)$ as a category over $\stinf(\E)$ via the identity map. Then, the space of lifts of $M_{\E}$ is exactly the category denoted $[M_{\E}]$ in Definition \ref{de:lifts}. Thus, $\pi_0(\L_{\stinf(\E)}(M_{\E}))$ is a point, $\pi_{t,1}(\L_{\stinf(\E)}(M_{\E})) \cong \gm_{\stinf(\E)}(M_{\E})$, and for all $s \geq 2$,
\begin{eqnarray*}
\pi_{s}(\L_{\stinf(\E)}(M_{\E}) & \cong & \pi_{s,0}(\hom_{\stinf(\E)}(M,M).
\end{eqnarray*}
\end{ex}

\begin{de} \label{de:liftable}
Let $\C$ be a category over $\stinf(\E)$.
We say that $M$ is liftable to $\C$ if $\L_{\C}(M)$ is non empty.
\end{de}

Before we can study further examples, we need the following result, whose interest to our approach should now be clear.

\begin{pro} \label{pro:liftlimits}
The functor that sends a category $\C$ over $\stinf(\E)$ to its space of lifts commutes with limits.
\end{pro}

\begin{proof}
It is the composite of the pullback diagram defining the lifts, which commutes with limits, and the functor which associates to an $\infty$-category its core, which also commutes with limits as it is a right adjoint.
\end{proof}

\begin{ex}
Consider $\stinf(A)$ as a category over $\stinf(\E)$ via the functor $U_{\E}$. Then, the space of lifts of $M_{\E}$ has the following homotopy groups: $\pi_0(\L_{\stinf(A)}(M_{\E}))$ is the set of equivalence classes of stable $A$-modules $M$ such that $U_{\E}M = M_{\E}$. If $M_{\E}$ is liftable, we can choose a base point $M \in \L_{\stinf(A)}(M_{\E})$, then taking cores in the pullback diagram
\begin{equation*}
\xymatrix{ [M]_{\tilde{\L}_{\C}(M_{\E})} \ar[r] \ar[d] & [M]_{\C} \ar[d] \\
[M_{\E}] \ar[r]_{=} & [M_{\E}], }
\end{equation*}
where $[M]_{\D}$ stands for the full $\infty$-subcategory of $\D$ with one object $M$ gives an equivalence of $\infty$-categories $[M]_{\tilde{\L}_{\C}(M_{\E})} \cong [M]_{\C}$. Thus $\pi_{1}(\L_{\stinf(A)}(M_{\E}), M) \cong \gm_{\stinf(A)}(M)$, and for all $n \geq 2$,
\begin{eqnarray*}
\pi_{s}(\L_{\stinf(A)}(M_{\E}), M) & \cong & \pi_{s,0}(\End_{\stinf(A)}(M,M)),
\end{eqnarray*}
the latter being computed in Section \ref{sec:poincare}.
\end{ex}

In particular, the determination of $\pi_0(\L_{\stinf(A)}(M_{\E}))$ is equivalent to our original question \ref{prob:lift}. We now set up the tools which are necessary to compute this set.

\begin{lemma} \label{lemma:looplifts}
Let $U_{\E} : \C \rightarrow \stinf(\E)$ be a stable symmetric monoidal $\infty$-category over $\stinf(\E)$ and $M \in \stinf(\E)$. \\
For any $n \in \Z$, there is a canonical weak equivalence of spaces
$$ \L_{\C}(U_{\E}M) \eq \L_{\C}(U_{\E} \Omega^n M).$$
\end{lemma}

\begin{proof}
The functors under consideration are functors of stable symmetric monoidal $\infty$-categories. The result follows by definition of $\L_{\C}$.
\end{proof}

\begin{rk} The construction $\L_{\C}(M_{\E})$ is not functorial in $M_{\E}$. However, we can use the monoidal structure to produce some arrows between lift spaces.
\end{rk}

\begin{pro} \label{cor:lunmonoid}  \label{pro:liftmodule}
Let $\C \rightarrow \stinf(\E)$ be a stable symmetric monoidal $\infty$-category over $\stinf(\E)$. \begin{enumerate}
\item For all liftable $M,N \in \stinf(\E)$, the tensor product induces a map
$$\L_{\C}^\otimes : \L_{\C}(M) \times \L_{\C}(N) \rightarrow \L_{\C}(M \otimes N).$$
\item The map $\L_{\C}^\otimes$ is associative, {\it~i.e.} the following diagram commutes
$$ \xymatrix{\L_{\C}(L) \times\L_{\C}(M) \times \L_{\C}(N) \ar[d]^{id \times \L_{\C}^\otimes } \ar[r]^{\quad \L_{\C}^\otimes \times id} & \L_{\C}(L\otimes M) \times \L_{\C}(N) \ar[d]^{\L_{\C}^\otimes} \\
\L_{\C}(L) \times\L_{\C}(M \otimes N) \ar[r]_{\L_{\C}^\otimes} & \L_{\C}(L \otimes M \otimes N). } $$
\item In particular, the grouplike $\infty$-groupoid $\L_{\C}(\un)$ is an associative monoid. This recovers the ring structure on the reduced Picard $\infty$-groupoid.
\item The construction $\L_{\C}$ takes values in $\L_{\C}(\mathbb{1})$-bimodules, and for all liftable diagrams of modules $M,N$, $\L_{\C}^{\otimes}$ induces a map
$$ \L_{\C}(M) \times_{\L_{\C}(\mathbb{1})} \L_{\C}(N) \rightarrow \L_{\C}(M \otimes N),$$
where $\L_{\C}(M) \times_{\L_{\C}(\mathbb{1})} \L_{\C}(N)$ stands for the coequalizer
$$ \L_{\C}(M) \times \L_{\C}(\mathbb{1}) \times \L_{\C}(N) \rightrightarrows \L_{\C}(M) \times \L_{\C}(N).$$
\end{enumerate}
\end{pro}

\begin{proof}
We prove the three points separately. \begin{enumerate}
\item This map is induced by the tensor product: it sends two lifts $m,n$ of $M$ and $N$ respectively on $m \otimes n$, which is a lift of $M \otimes N$ as $U_{\E}$ is monoidal by Proposition \ref{de:adjforget}. It does similarly on $n$-arrows.
\item The map $\L_{\C}^{\otimes}$ is associative by construction, and associativity of $\otimes$ in the category $\stinf(\E)$.
\item Apply the two first assertions to $M = N = \un$. The result follows.

\item The last assertion comes from the observation that $\L_{\C}(M \otimes N)$ coequalize the diagram
$$ \L_{\C}(M) \times \L_{\C}(\mathbb{1}) \times \L_{\C}(N) \rightrightarrows \L_{\C}(M) \times \L_{\C}(N)$$
where the two arrows are induced by the action, as the tensor product certainly coequalize
$M \otimes \un \otimes N \rightrightarrows M \otimes N$.
\end{enumerate}
\end{proof}

\begin{thm} \label{cor:lifttot}
There is an equivalence of spaces
$$ \L(M) \eq Tot\large{(} \L_{\bullet+1}(M)\large{)},$$
where $$\L_{n+1}(M) := \L_{Mod_{\stinf(A)}(T^{\otimes n+1})}(M),$$
 and the faces and degeneracies induced by the faces and degeneracies of the cosimplicial object
 \begin{equation*} 
 \stinf(A) \cong Tot\left(  \xymatrix{\stinf(E) \ar@<-.5ex>[r] \ar@<.5ex>[r] & \mod_{\stinf(E)}(T_E) \ar@<-.7ex>[r] \ar@<-.0ex>[r] \ar@<.7ex>[r]  & \mod_{\stinf(E)}(T_E^{\otimes 2}) \hdots}\right).
\end{equation*}
\end{thm}

\begin{proof}
This is a consequence of Proposition \ref{pro:liftlimits}, applied to the totalization given in Theorem \ref{thm:totalization}.
\end{proof}

\begin{cor} \label{thm:liftss}
Let $M_{\E}\in\stinf(\E)$. Let
\begin{equation}
E_1(\L)^{s,n} = (\pi_{s}\L_{\mod_{\stinf(A)}(T^{n+1})}(M_{\E})).
\end{equation}
There are  obstruction classes to the lift problem in $E_1^{s,s+1}$ for each $s \geq 1$. In particular, $M_{\E}$ is liftable whenever these groups are zero.

Suppose that $M_{\E}$ is a liftable module, and let $M \in \stinf(A)$ be such a lift. There is a Bousfield-Kan spectral sequence
$$E_1(\L)^{s,n} \Rightarrow \pi_{s-n}( \L_{\stinf(A)}(M_{\E}),M).$$

This spectral sequence is called the spectral sequence for the space of lifts.
\end{cor}

\begin{proof}
This is the Bousfield-Kan spectral sequence associated with the equivalence
$ \L(M) \steq Tot\large{(} \L_{\bullet+1}(M)\large{)}$.  We then apply the general machinery associated to such a Bousfield-Kan spectral sequence: the obstruction theory to lift a base point along partial totalizations and the associated Bousfield-Kan spectral sequence.
\end{proof}

As in the last section, this settles the machinery we need to set up the spectral sequence which resolves the Problem \ref{prob:lift}. However, this spectral sequence is not yet computable. We need two more ingredients to make the computations accessible. First, we need to compare this spectral sequence to a spectral sequence we know something about. This is done in Proposition \ref{pro:bkssendautm}. Then, we can use the pairing $\End(M) \times \End(N) \rightarrow \End(M \otimes N)$, and in particular this pairing for $N = \un$, to have some extra structure on the spectral sequence for the space of lifts, through this comparison  map. This is Proposition \ref{pro:pairinginss}.

\begin{pro} \label{pro:bkssendautm}
Let $M$ be a stable $A$-module. Let $M_{\E} = U_{\E}M$. There are spectral sequences
\begin{equation}
E_1^{s,t,n}(\End) \cong \ext_A^{-s,t}( M, T^{\otimes n+1} \otimes M) \Rightarrow \ext_A^{-s-n,t}(M,M), 
\end{equation}

and

\begin{equation}
E_1^{s,n}(\gm(M)) =  \pi_s(\gm_A( M,T^{\otimes {n+1}}\otimes M)) \Rightarrow \pi_{s-n}(\gm_A(M)).
\end{equation}

Moreover, the natural transformation $\gm \rightarrow \End$ provided by Lemma \ref{lemma:endMandlimits} induces a morphism of spectral sequences 
$$ E_1^{s,n}(\gm) \rightarrow E_1^{s,0,n}(\End)$$
which is an isomorphism in degrees $(s,n)$ for  $s \geq 1$ starting at the $E_1$-page.
\end{pro}

\begin{proof}
Apply the functor $\End_{\C}$ (respectively $\gm_{\C}$ to the pointed categories appearing in Theorem \ref{thm:totalization}, using the base point $M \in \stinf(A)$, and $T_{\E}^{\otimes n} \otimes M_{\E} \in \mod_{\stinf(\E)}(T_{\E}^{\otimes n+1})$. The desired spectral sequence is the associated Bousfield-Kan spectral sequence.

We conclude by naturality of inclusion of units $\gm \rightarrow \End$.
\end{proof}

\begin{rk} \label{rk:ceendm}
Before we go further, observe that the previous spectral sequence is again a sort of generalization of the Cartan Eilenberg spectral sequence. Indeed, in the particular case when $\E$ has only one element $E$, the spectral sequence computing extension groups reduces to a spectral sequence
\begin{equation} \label{eq:e1pageendm}
E_1^{s,t,n}(\End(M)) \cong N^n\ext_E^{-s,t}( U_E M, T^{\otimes \bullet} \otimes U_EM) \Rightarrow \ext_A^{-s-n,t}(M,M). 
\end{equation}
\end{rk}

Although the $E_1$-page of this spectral sequence does seems as mysterious as its abutment, note that it only relies on the stable category of $E$-modules, for quasi-elementary $E$. Indeed, the $E_1$-page does only depend on the stable module $U_{\E}M$ (Remark \ref{rk:ceendm} is a special case of this fact). Moreover, this spectral sequence computes the homotopy groups of the space of lifts. Firstly, both its $E_1$-term and abutment are completely computable in principle for any finite dimensional Hopf algebra $A$ (use a cobar complex to do the explicit computation for example) giving information about its differentials. Secondly, this spectral sequence compares to the spectral sequence for the space of lifts, so that we get some differentials for free in the latter. The comparison map is given in the following proposition.

\begin{pro} \label{pro:comparisonliftendm}
Via the canonical identification $\pi_{i-1}(\gm_{(-)}(M_{-})) \eq \pi_i(\L_{-}(M_{-}))$, for $i \geq 1$, there is an isomorphism vector spaces
\begin{equation}
E^{s-1,n}_r(\gm) \cong E_r^{s,n}(M_{\E}).
\end{equation}
Moreover, the differentials originating at $E_r^{s,n}$ for $s-n \geq 1$ and $r \leq s$ in the descent spectral sequence converging to $\pi_{s-n}(\L_{\stinf(A)}(M_{\E}))$ are identified with the ones of $E^{s-1,n}_r(\gm)$.
\end{pro}

\begin{proof}
The $E_1$-page deduced by construction of the spectral sequence, using 
\begin{equation*}
\pi_{i-1}(\gm_{(-)}(M_{-})) \eq \pi_i(\L_{-}(M_{-})).
\end{equation*}
The comparison of differentials is given by \cite[Section 5]{MS14}).
\end{proof}

Finally, here is the last theoretical tool we need to carry-on computations: the action of $\End(\un)$ on $\End(M)$.

\begin{pro} \label{pro:pairinginss}
Let $\C$ be a pointed stable symmetric monoidal $\infty$-category  with distinguished object $M_{\C}$.
There are pairings
\begin{equation*}
\End_{\C}(\un) \wedge \End_{\C}(M_{\C}) \rightarrow \End_{\C}(M_{\C})
\end{equation*}
and 
\begin{equation*}
\gm_{\C}(\un) \wedge \gm_{\C}(M_{\C}) \rightarrow \gm_{\C}(M_{\C}).
\end{equation*}
These induces pairings of spectral sequences
\begin{equation*}
E_1^{s,t,n}(\End) \otimes E_1^{s',t',n'}(\End(M)) \rightarrow E_1^{s+s',t+t',n+n'}(\End(M))
\end{equation*}
and
\begin{equation*}
E_1^{s,n}(\gm) \otimes E_1^{s',n'}(\gm(M)) \rightarrow E_1^{s+s',n+n'}(\gm(M))
\end{equation*}
converging to the evident pairing in homotopy.
\end{pro}

\begin{proof}
This module structure comes from the general pairings in Bousfield-Kan spectral sequences.
\end{proof}

\part{Applications: $\A(1)$-modules revisited} \label{part:applications}

As an application, we show how to use the techniques developed in Part $1$ and $2$ to the study of the stable category of $\A(1)$-modules. In this setting, we recover the computation of the Picard group of the stable category of $\A(1)$-modules which is due to Adams and Priddy in \cite{AP76}. This was particularly important in the proof of their uniqueness theorem.

Using our technique, it is also possible to recover and extend the classification of lifts of $\E(1)$-modules to $\A(1)$-modules from \cite{Pow15}. However, in \textit{loc cit}, the author needed an indecomposability hypothesis which is not needed for our approach. The computation being quite tedious, we only show how this applies in a couple of examples at the end of this part.

The plan of this part follows the study made in Part \ref{part:descent} in the special case $A = \A(1)$. We first compute $\ext_{\A(1)}^{s,t}(\F,\F)$, and use Poincar\'e duality (see section \ref{sec:poincare}) to identify $\pi_{s,t}(\End_{\A(1)}(\F)$. This gives the abutment of the spectral sequence for the endomorphism space of the unit, and we can retro-engineer it to find some of the differentials in this spectral sequence. Finally, the comparison between the spectral sequence for the endomorphism space of the units and the descent spectral sequence for the Picard space gives us all the differentials in the latter spectral sequence that we need to conclude. The conclusion of this computation is Theorem \ref{thm:pica1}, which is originally due to Adams and Priddy.

\section{Generalities about $\A(1)$-modules}

The Hopf algebra $\A(1)$ is a finite dimensional cocommutative Hopf algebra over $\F$. It is classically defined as the sub-Hopf algebra of the modulo $2$ Steenrod algebra generated by $Sq^1$ and $Sq^2$. For our purposes, it will be more convenient to consider it using generators and relations. Its presentation is given in the following definition.

\begin{de}
Let $\A(1)$ be the sub-Hopf algebra of the modulo $2$ Steenrod algebra generated by the Steenrod squares $Sq^1$ and $Sq^2$. As an algebra, 
$$\A(1) \eq \frac{\F\langle Sq^1,Sq^2\rangle}{(Sq^1)^2, (Sq^2)^2 + Sq^1Sq^2Sq^1}.$$
Let $\E(1)$ be the sub-Hopf algebra of $\A(1)$ generated by the two first Milnor operations $Q_0 = Sq^1$ and $Q_1 = Sq^1 Sq^2 + Sq^2 Sq^1$. It is a primitively generated exterior algebra on $Q_0$ and $Q_1$.
\end{de}

The algebra $\A(1)$ is represented graphically in figure \ref{fig:a1}.

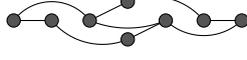
\begin{figure}
\definecolor{qqqqff}{rgb}{0.3333333333333333,0.3333333333333333,0.3333333333333333}
\begin{tikzpicture}[line cap=round,line join=round,>=triangle 45,x=0.5cm,y=0.5cm]
\clip(4.3235580728719665,10.790907037910573) rectangle (11.444061329611309,13.531444373837745);
\draw (6.,7.)-- (7.,7.);
\draw (9.,7.)-- (10.,7.);
\draw [shift={(7.,6.25)}] plot[domain=0.6435011087932844:2.498091544796509,variable=\t]({1.*1.25*cos(\t r)+0.*1.25*sin(\t r)},{0.*1.25*cos(\t r)+1.*1.25*sin(\t r)});
\draw [shift={(9.,6.25)}] plot[domain=0.6435011087932844:2.498091544796509,variable=\t]({1.*1.25*cos(\t r)+0.*1.25*sin(\t r)},{0.*1.25*cos(\t r)+1.*1.25*sin(\t r)});
\draw [shift={(8.,7.75)}] plot[domain=3.7850937623830774:5.639684198386302,variable=\t]({1.*1.25*cos(\t r)+0.*1.25*sin(\t r)},{0.*1.25*cos(\t r)+1.*1.25*sin(\t r)});
\draw (5.,12.)-- (6.,12.);
\draw (7.,12.)-- (8.,12.5);
\draw (8.,11.5)-- (9.,12.);
\draw (10.,12.)-- (11.,12.);
\draw [shift={(6.,11.25)}] plot[domain=0.6435011087932844:2.498091544796509,variable=\t]({1.*1.25*cos(\t r)+0.*1.25*sin(\t r)},{0.*1.25*cos(\t r)+1.*1.25*sin(\t r)});
\draw [shift={(8.,14.4)}] plot[domain=4.31759786068493:5.10718010008445,variable=\t]({1.*2.6*cos(\t r)+0.*2.6*sin(\t r)},{0.*2.6*cos(\t r)+1.*2.6*sin(\t r)});
\draw [shift={(10.,13.05)}] plot[domain=3.9513762261599594:5.4734017346094195,variable=\t]({1.*1.45*cos(\t r)+0.*1.45*sin(\t r)},{0.*1.45*cos(\t r)+1.*1.45*sin(\t r)});
\draw [shift={(7.335714285714286,13.092857142857145)}] plot[domain=3.82732216449608:5.107498470019571,variable=\t]({1.*1.7258242632884464*cos(\t r)+0.*1.7258242632884464*sin(\t r)},{0.*1.7258242632884464*cos(\t r)+1.*1.7258242632884464*sin(\t r)});
\draw [shift={(8.664285714285713,10.907142857142857)}] plot[domain=0.685729510906286:1.9659058164297778,variable=\t]({1.*1.7258242632884444*cos(\t r)+0.*1.7258242632884444*sin(\t r)},{0.*1.7258242632884444*cos(\t r)+1.*1.7258242632884444*sin(\t r)});
\begin{scriptsize}
\draw [fill=qqqqff] (6.,7.) circle (2.5pt);
\draw [fill=qqqqff] (7.,7.) circle (2.5pt);
\draw [fill=qqqqff] (8.,7.) circle (2.5pt);
\draw [fill=qqqqff] (9.,7.) circle (2.5pt);
\draw [fill=qqqqff] (10.,7.) circle (2.5pt);
\draw [fill=qqqqff] (5.,12.) circle (2.5pt);
\draw [fill=qqqqff] (6.,12.) circle (2.5pt);
\draw [fill=qqqqff] (7.,12.) circle (2.5pt);
\draw [fill=qqqqff] (9.,12.) circle (2.5pt);
\draw [fill=qqqqff] (10.,12.) circle (2.5pt);
\draw [fill=qqqqff] (11.,12.) circle (2.5pt);
\draw [fill=qqqqff] (8.,12.5) circle (2.5pt);
\draw [fill=qqqqff] (8.,11.5) circle (2.5pt);
\end{scriptsize}
\end{tikzpicture}

\caption{The subalgebra $\A(1)$ of the Steenrod algebra. Each dot represents a copy of $\F$. Straight lines represent the action of $Sq^1$ and curved ones represent the action of $Sq^2$.} \label{fig:a1}
\end{figure}

Our methods are particularly efficient for the study of the stable module category of $\A(1)$, as there is only one maximal quasi-elementary sub-Hopf algebra of $\A(1)$: the exterior algebra on the first two Milnor operations $Q_0$ and $Q_1$, denoted $\E(1)$.

In particular, we know by Proposition \ref{pro:detection} that an $\A(1)$-module is in $Pic(\A(1))$ if and only if its restriction to $\E(1)$ is. For instance, the joker (see figure \ref{fig:joker}) is a non-trivial element of the Picard group of $\A(1)$, as its restriction to $\E(1)$ is stably isomorphic to $un$.

\begin{figure}
\definecolor{qqqqff}{rgb}{0.3333333333333333,0.3333333333333333,0.3333333333333333}
\begin{tikzpicture}[line cap=round,line join=round,>=triangle 45,x=0.5cm,y=0.5cm]
\clip(5.82942940857475,5.953777759762861) rectangle (10.537605207779377,8.02281283886911);
\draw (6.,7.)-- (7.,7.);
\draw (9.,7.)-- (10.,7.);
\draw [shift={(7.,6.25)}] plot[domain=0.6435011087932844:2.498091544796509,variable=\t]({1.*1.25*cos(\t r)+0.*1.25*sin(\t r)},{0.*1.25*cos(\t r)+1.*1.25*sin(\t r)});
\draw [shift={(9.,6.25)}] plot[domain=0.6435011087932844:2.498091544796509,variable=\t]({1.*1.25*cos(\t r)+0.*1.25*sin(\t r)},{0.*1.25*cos(\t r)+1.*1.25*sin(\t r)});
\draw [shift={(8.,7.75)}] plot[domain=3.7850937623830774:5.639684198386302,variable=\t]({1.*1.25*cos(\t r)+0.*1.25*sin(\t r)},{0.*1.25*cos(\t r)+1.*1.25*sin(\t r)});
\begin{scriptsize}
\draw [fill=qqqqff] (6.,7.) circle (2.5pt);
\draw [fill=qqqqff] (7.,7.) circle (2.5pt);
\draw [fill=qqqqff] (8.,7.) circle (2.5pt);
\draw [fill=qqqqff] (9.,7.) circle (2.5pt);
\draw [fill=qqqqff] (10.,7.) circle (2.5pt);
\end{scriptsize}
\end{tikzpicture}

\caption{The joker. Each dot represents a copy of $\F$. Straight lines represent the action of $Sq^1$ and curved ones represent the action of $Sq^2$.} \label{fig:joker}
\end{figure}
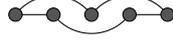

Using an explicit resolution of $\F$ as an $\A(1)$-module (or the Cartan-Eilenberg spectral sequence associated to $\E(1) \subset \A(1)$, or the May spectral sequence for $\A(1)$), it is easy to compute $\ext_{\A(1)}^{s,t}(\F,\F)$. The alert reader will recognise the $E_2 = E_{\infty}$-page of the Adams spectral sequence computing $kO_*$. Explicitly,
\begin{equation} \label{eq:ha1}
\ext_{\A(1)}^{s,t}(\F,\F) \cong  \frac{\F[v_0, \eta, \alpha, \beta]}{(v_0\eta, \eta^3, \eta \alpha, \alpha^2-v_0^2\beta)},
\end{equation}
with $|v_0| = (1,1)$, $|\eta| = (1,2)$, $|\alpha| = (3,7)$, and $|\beta| = (4,12)$.

In particular, Poincar\'e duality gives the homotopy groups of the space $\End_{\A(1)}(\F)$:
\begin{equation}
\pi_{s,t}(\End_{\A(1)}(\F)) \cong \ext_{\A(1)}^{s-1,-6-t}(\F,\F).
\end{equation}

\section{The spectral sequence for the endomorphism space of the unit}

When $A =\A(1)$, the spectral sequence provided by Proposition \ref{pro:bkssendaut} reads
\begin{equation} \label{eq:ssenda1}
E_1^{s,t,n}(\End) \cong N^n\ext_{\A(1)}^{-s,t}(\F,T^{\otimes \bullet+1})^{\ast} \Rightarrow \ext_{\A(1)}^{s-n-1,-6-t}(\F,\F)^{\ast},
\end{equation}
where $T = \left( \A(1)\sur \E(1) \right)^{\ast} = E(\xi_1^2)$, the exterior algebra on $\xi_1^2$ in degree $-2$.

But we are in the easy case when there is only one elementary Hopf subalgebra of $\A(1)$, so remark \ref{rk:simplicifationoneelementary} applies, giving a reformulation of the $E_1$-page of the spectral sequence \eqref{eq:ssenda1}

\begin{equation} \label{eq:refssenda1}
E_1^{s,t,n}(\End) \cong N^n(\ext_{\E(1)}^{-s,t}(\F,T^{\otimes \bullet}))^* \Rightarrow (\ext_{\A(1)}^{s-n-1,-6-t}(\F,\F))^*.
\end{equation}

We can actually compute explicitly the first page of this spectral sequence.

\begin{lemma}
The $E_2$-page of the spectral sequence 
\begin{equation} \label{eq:refssenda1e2}
E_1^{s,t,n}(\End) \cong \Sigma^{-1,-4} \F[v_0^{-1},v_1^{-2} \oplus \F[v_1^{-2},\eta] \Rightarrow (\ext_{\A(1)}^{s-n-1,-6-t}(\F,\F))^*,
\end{equation}
where $|v_0^{-1} | = (1,-1,0)$, $|v_1^{-1}| = (1,-3,0)$ and $|\eta| = (0,-2,1)$.
\end{lemma}

There is a slight abuse of notation in the previous lemma. Indeed, the element $v_0^{-1}$ is not an actual inverse to $v_0$, but is related to the dual of $v_0$ in $\ext_{\A(1)}(\F,\F)$.

\begin{proof}
By construction, the differential on the $E_1$-page is induced by the alternating sum of the maps
$$id \otimes \mu \otimes id : T^{\otimes n} \rightarrow T^{\otimes n-1},$$
where $\mu : T^{\otimes 2} \rightarrow T$ is the multiplication of the commutative ring object $T$.

But $T$ is concentrated in even degrees, so, as an $\E(1)$-module, it is trivial. In particular, 

$$\ext_{\E(1)}^{s-1,-4-t}(\F,T^{\otimes n}) = \ext_{\E(1)}^{s-1,-4-t}(\F,\F) \otimes T^{\otimes n}$$

Now, $T$ being an exterior algebra on $\xi_1^2$ of degree $-2$, the $E_1$-page is the tensor product of a minimal resolution of $\un$ over the exterior algebra $T$ with $\ext_{\E(1)}(\un,\un)$. The action of $T$ on this extension group is coming from the diagonal, indeed, the action of $T$ is given by the coaction of $(\xi^2)^{\ast}$ on $\hom_{\E(1)}(\A(1),k)$. 
Since $\Delta(\xi_2) = 1 \otimes \xi_2 + \xi_1^2 \otimes \xi_1 + \xi_2 \otimes 1$, we have $\xi_1^2 v_0^{-i}v_1^{2j+1} = v_0^{i+1} v_1^{-2j}$. The asserted $E_1$-page follows.
\end{proof}

Finally, the only differentials we are interested in will follow from our knowledge of the abutment.

\section{The stable Picard group of $\A(1)$}

We are now ready to compute the Picard group of the stable category of $\A(1)$-modules. As stated in the introduction, this follows by comparison with the spectral sequence for the space of endomorphisms of the unit.

\begin{thm} \label{thm:pica1}
The stable Picard group of $\A(1)$ is isomorphic to $\Z \oplus \Z \oplus \Z/2$. A representative of the element corresponding to $a,b,c \in \Z \oplus \Z \oplus \Z/2$  is 
$\Omega^a J^{c}[b]$, where $J$ is the joker.
\end{thm}

The proof is essentially an inspection of the spectral sequence for the Picard space. We divide this inspection into a series of lemmas.

Recall that the comparison result given in Proposition \ref{pro:comparison}, we have an isomorphism
\begin{equation}
E_1^{s-1,0,n}(\End) \cong E_1^{s,n}(\pic)
\end{equation}
whenever $n \geq 0$ and $s \geq 2$. Thus, our first task is to determine $E_2^{s-1,t,n}(\End)$, for $t = 0$.

\begin{lemma} \label{lemma:upperboundonpic}
At the $E_{\infty}$-page, $E_{\infty}^{s,s-1}$ is a subquotient of $\Z \oplus \Z \oplus \Z/2$.
\end{lemma}

\begin{proof}
We have to determine $E_1^{s,0,s+1}(\End)$.

The degree of a generic element $\Sigma^{-1,-4}v_1^{-2i}\eta^j$ is $(2i,-6i+2j,j)$. The constraint on the tridegree gives $j = i+2$, and thus $-2i + 4 = 4$. There is only one element here, namely $\Sigma^{-1,-4}\eta^2$.

 The comparison result of Proposition \ref{pro:comparison} gives us an isomorphism
$$ E_2^{s,s-1}(\pic) \cong E_2^{s-1,0,s-1}(\End)$$

for $s\geq 2$. Moreover, the values of $s$ that are left out by the previous identification are already known from the $E^1$-page: $E_2^{0,1}(\pic) = \pi_0(\pic(\E(1)))= \Z \oplus \Z$, as this is the classical Picard group of the stable category of $\E(1)$-modules, and  $E_2^{1,0}(\pic) \cong E_2^{0,0,0}(\End)$ is reduced to a point.
\end{proof}

\begin{proof}[proof of Theorem \ref{thm:pica1}]
We know that at least one element has to survive because of the joker, hence we have a complete determination.
\end{proof}

\begin{rk}
The reader might have noticed that the Picard group of $\E(1)$ did not actually play any role in this computation. In particular, we could have focused on the relative Picard group of $\E(1) \subset \A(1)$, that is, the kernel of the group homomorphism $\pi_0(\pic(\A(1))) \rightarrow \pi_0(\pic(\E(1))) \cong \Z \oplus \Z$. This is indeed possible, and the computation is taken care of by the spectral sequence for the spaces of lifts. Indeed, classes of elements $[M]$ in the relative Picard group admits exactly two representatives which are lifts of $\un \in \st(\E(1))$, namely $\un$ and the joker.

We leave the computation of the relative Picard group by the machinery developed in section \ref{sec:descentlift} to the interested reader (the appropriate spectral sequences coincides with the spectral sequence for Picard spaces in a wide range of degree, so the arguments are completely similar).
\end{rk}

\section{A study of some lifts}

We conclude our study of $\A(1)$-modules by solving a couple of lifting problems. Note that it actually is possible to classify the lifts of all stably indecomposable $\E(1)$-modules using this, and that the technology we developed also applies to stably decomposable $\E(1)$-modules as well. However, we leave to the reader the bookkeeping needed to study all the different cases to focus on two examples.

\subsection{An obstruction to existence}

Our first example of application is the computation of the obstruction to lift the $\E(1)$-module $M = k \{ m_0,m_2,m_3,m_4,m_5,m_7 \}$, with action $Q_0 m_2 = m_3$, $Q_0m_4 = m_5$, $Q_1m_0 = m_3$, $Q_1m_2 = m_5$, and $Q_1m_4 = m_7$, where $m_i$ is in degree $i$ (see figure \ref{fig:nonliftablea1}  for a picture of this module).

It turns out that there is a non trivial obstruction surviving the spectral sequence. This is for a good reason: $M$ does not lift to $\A(1)$.

\begin{figure}
\definecolor{qqqqff}{rgb}{0.3333333333333333,0.3333333333333333,0.3333333333333333}
\begin{tikzpicture}[line cap=round,line join=round,>=triangle 45,x=0.5cm,y=0.5cm]
\clip(1.1727272727272733,2.8418181818181805) rectangle (10.063636363636363,7.114545454545453);
\draw (2.,4.)-- (5.,4.);
\draw (5.,4.)-- (4.,5.);
\draw (4.,5.)-- (7.,5.);
\draw (1.9909090909090914,4.223636363636363) node[anchor=north west] {$m_0$};
\draw (4.00909090909091,5.223636363636363) node[anchor=north west] {$m_2$};
\draw (5.009090909090909,4.223636363636363) node[anchor=north west] {$m_3$};
\draw (7.009090909090909,5.223636363636363) node[anchor=north west] {$m_5$};
\draw (7.,5.)-- (6.,6.);
\draw (6.,6.)-- (9.,6.);
\draw (6.009090909090909,6.223636363636363) node[anchor=north west] {$m_4$};
\draw (9.009090909090908,6.223636363636363) node[anchor=north west] {$m_7$};
\begin{scriptsize}
\draw [fill=qqqqff] (2.,4.) circle (2.5pt);
\draw [fill=qqqqff] (5.,4.) circle (2.5pt);
\draw [fill=qqqqff] (4.,5.) circle (2.5pt);
\draw [fill=qqqqff] (7.,5.) circle (2.5pt);
\draw [fill=qqqqff] (6.,6.) circle (2.5pt);
\draw [fill=qqqqff] (9.,6.) circle (2.5pt);
\end{scriptsize}
\end{tikzpicture}

\caption{A graphical representation of $M$} \label{fig:nonliftablea1} 
\end{figure}
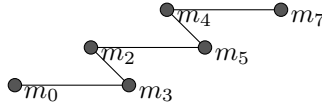

\begin{lemma} \label{lemma:mperiodic}
There is an isomorphism
\begin{equation}
\Omega M \cong \Sigma^{-1} M.
\end{equation}
\end{lemma}

\begin{proof}
The injective enveloppe of $M$ is
\begin{equation*}
M \rightarrow \Sigma^{-1} \E(1) \oplus \Sigma^{1} \E(1) \oplus \Sigma^{3} \E(1).
\end{equation*}
The cokernel of this map is exactly $\Sigma^{-1}M$. By definition, this is $\Omega M$.
\end{proof}

The extension groups of $M$ follows.

\begin{lemma}
There is an isomorphism
\begin{equation}
\ext_{\E(1)}^{s,t}(\un,M) \cong \F[v_0^{\pm 1}] \{x_{-3},x_{-5},x_{-7} \} 
\end{equation}
where $|x_i| = i $.
\end{lemma}

\begin{proof}
One has to compute the vector space of maps $\un \rightarrow M$, which is precisely $\F \{x_{-3},x_{-5},x_{-7} \}$, and conclude by the periodicity provided by lemma \ref{lemma:mperiodic}.
\end{proof}

Now, observe that this is enough to determine the spectral sequence for the space of endomorphisms of $M$. Indeed,
\begin{equation*}
\ext_{\E(1)}^{s,t}(M,M) \cong \ext_{\E(1)}^{s,t}(\un, M^{\ast} \otimes M) \cong \ext_{\E(1)}^{s,t}(\un, M \oplus \Sigma^{-7}M)
\end{equation*}
where $M^{\ast}$ is the linear dual of $M$. The stable isomorphism $M^{\ast} \otimes M \cong M \oplus \Sigma^{-7} M$ is a straightforward computation.

\begin{pro} \label{pro:e2pagenonliftable}
The spectral sequence for the space of endomorphisms of $M$ has $E_2$-page:
\begin{equation}
E_2^{s,t}(\End) = \F[v_0^{-1}, \eta] \{ x_{-7},x_{0} \}.
\end{equation}
In particular, \begin{itemize}
\item a generators for $E^{s-1,0,s}(\End)$ is
$\eta^6v_0^{-5}x_{-7}$,
\item a generators for $E^{s-2,0,s}(\End)$ is 
$\eta^5v_0^{-3}x_{-7}$,
\item $E^{s-3,0,s}(\End)$ is of dimension $1$, generated by $\eta^{4}v_0^{-1}x_{-7}$,
\item and $E^{s-4,0,s}(\End)$ is trivial.
\end{itemize}
\end{pro}

\begin{proof}
The computation of the $E_1$-page is given by the identification of the $E_1$-page provided by \eqref{eq:e1pageendm}, together with the observation made in the proof of the identification \eqref{eq:refssenda1e2}: for degree reasons $T$ has a trivial action of $\E(1)$, so the first differential is induced by the multiplication of $T$ on 
\begin{equation*}
 \F[v_0^{-1}] \{ x_{-7},x_{-5},x_{-3},x_{0},x_{2},x_{4} \} \otimes T^{\otimes n}.
\end{equation*}

Now, the action of $T$ can be determined in the following way: the element $x_{-3}$ has a trivial action of $\xi_1^2$ for degree reasons. Because of the action of the Picard space on the space of lifts, the formula $v_1^{-1}x_{-3} = v_0^{-1} x_{-5}$ gives the action $\xi_1^2 v_0^{-1}x_{-5} = v_0^{-1}x_{-3}$, and finally $\xi_1^2 x_{-5} = x_{-3}$. This gives the asserted $E_1$-page.

Now, the non-trivial element in degrees $(s-1,0,s),(s-2,0,s)$ and $(s-3,0,s)$ follows from inspection, as a generic element of the $E_2$-page has tridegree
\begin{equation*}
|v_0^i\eta^jx_k| = (-i,i + 2j +k,j).
\end{equation*}
\end{proof}

\begin{cor}
There is a non-trivial obstruction to the realization of $M$ as an $\A(1)$-module, namely $\eta^3v_0^{-1}x_{-5}$.
\end{cor}

\begin{proof}
Proposition \ref{pro:comparisonliftendm} provides an identification of the spectral sequence for the space of lifts for $M$ with the spectral sequence  for the endomorphism space, namely
\begin{equation}
E_r^{s-1,0,n}(\End) \cong E_r^{s,0,n}(\L)
\end{equation}
Thus, the obstruction to existence of a lift, living in degree $(s-1,s)$ in $E_1^{s,0,n}(\L)$ can be seen in $E_1^{s-2,0,n}(\End)$.

The asserted element is in degree $n = 3$, so it cannot support any differential $d_r$ for $r \geq 2$, for degree reasons,  as its potential target, namely $\eta^4v_0^-1x_{-7}$ could only have been hit by a $d_1$.
Similarly, it cannot be hit by any differential since the element that is a potential source of a differential hitting this elements, namely  $\eta^6v_0^{-5}x_{-7}$, have spectral sequence degree $6$. Thus $\eta^5v_0^{-3}x_{-7}$ survives to $E_{\infty}$.
\end{proof}

\subsection{Multiple lifts for an $\E(1)$-module}

After studying an $\E(1)$-module which is not liftable to $\A(1)$, we turn to the other extreme. In this section, we will show using the spectral sequence for the space of lifts that a certain $\E(1)$-module admits at most $8$ different lifts as an $\A(1)$-module. 

In this case, the estimate given by the spectral sequence for the space of lifts is sharp, since we can actually find $8$ non-isomorphic $\A(1)$-modules that solve the lifting problem.

Let $N$ be the $\E(1)$-module $\F \{ n_{-1},n_{0},n_{1},n_{2} \}$, where $n_k$ is in degree $k$, and the actions are $Q_0n_{-1} = n_0$, $Q_0 n_1 = n_2$, and $Q_1 n_{-1} = n_2$ (see figure \ref{fig:liftablea1} for a graphical representation of this module).

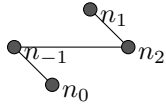
\begin{figure}
\definecolor{qqqqff}{rgb}{0.3333333333333333,0.3333333333333333,0.3333333333333333}
\begin{tikzpicture}[line cap=round,line join=round,>=triangle 45,x=0.5cm,y=0.5cm]
\clip(3.5181818181818185,10.150909090909092) rectangle (10.063636363636363,13.96909090909091);
\draw (2.,4.)-- (5.,4.);
\draw (5.,4.)-- (4.,5.);
\draw (4.,5.)-- (7.,5.);
\draw (1.9909090909090914,4.223636363636362) node[anchor=north west] {$m_0$};
\draw (4.00909090909091,5.223636363636363) node[anchor=north west] {$m_2$};
\draw (5.009090909090909,4.223636363636362) node[anchor=north west] {$m_3$};
\draw (7.009090909090909,5.223636363636363) node[anchor=north west] {$m_5$};
\draw (7.,5.)-- (6.,6.);
\draw (6.,6.)-- (9.,6.);
\draw (6.009090909090909,6.223636363636363) node[anchor=north west] {$m_4$};
\draw (9.009090909090908,6.223636363636363) node[anchor=north west] {$m_7$};
\draw (5.,12.)-- (6.,11.);
\draw (5.,12.)-- (8.,12.);
\draw (7.,13.)-- (8.,12.);
\draw (5.009090909090909,12.223636363636365) node[anchor=north west] {$n_{-1}$};
\draw (6.009090909090909,11.223636363636365) node[anchor=north west] {$n_{0}$};
\draw (7.009090909090909,13.223636363636365) node[anchor=north west] {$n_1$};
\draw (8.009090909090908,12.223636363636365) node[anchor=north west] {$n_2$};
\begin{scriptsize}
\draw [fill=qqqqff] (2.,4.) circle (2.5pt);
\draw [fill=qqqqff] (5.,4.) circle (2.5pt);
\draw [fill=qqqqff] (4.,5.) circle (2.5pt);
\draw [fill=qqqqff] (7.,5.) circle (2.5pt);
\draw [fill=qqqqff] (6.,6.) circle (2.5pt);
\draw [fill=qqqqff] (9.,6.) circle (2.5pt);
\draw [fill=qqqqff] (5.,12.) circle (2.5pt);
\draw [fill=qqqqff] (6.,11.) circle (2.5pt);
\draw [fill=qqqqff] (8.,12.) circle (2.5pt);
\draw [fill=qqqqff] (7.,13.) circle (2.5pt);
\end{scriptsize}
\end{tikzpicture}
\caption{The $\E(1)$-module $N$} \label{fig:liftablea1}
\end{figure}

\begin{lemma}
There is an isomorphism
\begin{equation}
\Omega N \cong \Sigma^{-3} N.
\end{equation}
\end{lemma}

\begin{proof}
The injective enveloppe of $N$ is
\begin{equation*}
N \rightarrow \Sigma^{-3} \E(1) \oplus \Sigma^{-1} \E(1).
\end{equation*}
The cokernel of this map is exactly $\Sigma^{-3}N$. By definition, this is $\Omega N$.
\end{proof}

Again, this periodicity has a direct consequence on extension groups.

\begin{lemma}
There is an isomorphism
\begin{equation}
\ext_{\E(1)}^{s,t}(\un,N) \cong \F[v_1^{\pm 1}] \{ x_{-2},x_0 \},
\end{equation}
where $|x_i| = i$.
\end{lemma}

\begin{proof}
One computes $\ext_{\E(1)}^{0,t}(\un,N) \cong \F \{x_{-2},x_0 \}$ and concludes by periodicity.
\end{proof}

The previous computation is enough to compute the spectral sequence for the space of endomorphisms. Indeed
\begin{equation*}
\ext_{\E(1)}^{s,t}(N,N) \cong \ext_{\E(1)}^{s,t}(\un, N^{\ast} \otimes N) \cong \ext_{\E(1)}^{s,t}(\un, N \oplus \Sigma^{-1}N),
\end{equation*}
where $N^{\ast}$ is the linear dual of $N$.

\begin{pro}
The spectral sequence for the space of endomorphisms of $N$ has $E_2$-page:
\begin{equation}
E_2^{s,t}(\End)\cong \F[v_1^{-1},\eta] \{ x_{-2},x_{-1},x_{0},x_1 \}.
\end{equation}
In particular, the generators in degree $(s-1,0,s)$ are $\eta^2 v_1^{-1}x_{-1}$, $\eta^3 v_1^{-2}x_{0}$, and $\eta^4 v_1^{-3}x_{1}$.
\end{pro}

\begin{proof}
The proof is similar to the proof of Proposition \ref{pro:e2pagenonliftable}.

The computation of the $E_2$-page is given by the identification of the $E_1$-page provided by \eqref{eq:e1pageendm}, together with the observation made in the proof of the identification \eqref{eq:refssenda1e2}: for degree reasons $T$ has a trivial action of $\E(1)$, so the first differential is induced by the multiplication of $T$ on 
\begin{equation*}
 \F[v_1^{-1}] \{ x_{-2},x_{-1},x_{0},x_{1} \} \otimes T^{\otimes n}.
\end{equation*}
The identification of the non-trivial elements in the appropriate degree comes from the explicit formula for the degree of a generic element in this spectral sequence:
\begin{equation*}
|v_1^i\eta^jx_k| = (-i,3i+2j+k,j).
\end{equation*}
The result follows by inspection.
\end{proof}

\begin{cor}
There are at most $8$ $\A(1)$-modules whose underlying $\E(1)$-module is $N$.
\end{cor}

Finally, one observes that this bound is actually sharp in our case, since the $8$ different $\A(1)$-modules described graphically in figure \ref{fig:8liftsofn} are solutions to the lifting problem for $N$.

\begin{figure}
\definecolor{qqqqff}{rgb}{0.3333333333333333,0.3333333333333333,0.3333333333333333}
\begin{tikzpicture}[line cap=round,line join=round,>=triangle 45,x=0.5cm,y=0.5cm]
\clip(35.59591569613088,-8.232340631622867) rectangle (51.7747216408982,7.994473639390187);
\draw (37.54285071136855,2.6904768936536914)-- (37.54285071136855,1.6904768936536914);
\draw (40.54285071136855,2.6904768936536914)-- (39.54285071136855,1.6904768936536914);
\draw (40.54285071136855,4.690476893653692)-- (40.54285071136855,3.6904768936536914);
\draw (40.54285071136855,1.6904768936536914)-- (40.54285071136855,0.6904768936536914);
\draw [shift={(41.303801629353885,3.6904768936536914)}] plot[domain=2.221269262549897:4.061916044629689,variable=\t]({1.*1.256601090077007*cos(\t r)+0.*1.256601090077007*sin(\t r)},{0.*1.256601090077007*cos(\t r)+1.*1.256601090077007*sin(\t r)});
\draw [shift={(41.303801629353885,1.6904768936536911)}] plot[domain=2.221269262549897:4.061916044629689,variable=\t]({1.*1.2566010900770068*cos(\t r)+0.*1.2566010900770068*sin(\t r)},{0.*1.2566010900770068*cos(\t r)+1.*1.2566010900770068*sin(\t r)});
\draw [shift={(39.82346582424781,2.6904768936536922)}] plot[domain=-0.9471784979433826:0.9471784979433822,variable=\t]({1.*1.2318744318386192*cos(\t r)+0.*1.2318744318386192*sin(\t r)},{0.*1.2318744318386192*cos(\t r)+1.*1.2318744318386192*sin(\t r)});
\draw (40.54285071136855,0.6904768936536914)-- (40.54285071136855,1.6904768936536914);
\draw (40.54285071136855,3.6904768936536914)-- (40.54285071136855,4.690476893653692);
\draw [shift={(36.72823549855945,2.7027660511488283)}] plot[domain=-0.8931804282524105:0.908390767695257,variable=\t]({1.*1.2993564881595165*cos(\t r)+0.*1.2993564881595165*sin(\t r)},{0.*1.2993564881595165*cos(\t r)+1.*1.2993564881595165*sin(\t r)});
\draw (37.52735981276525,3.7273282082517265)-- (37.52735981276525,4.727328208251727);
\draw [shift={(40.30253799001807,2.7147391196253197)}] plot[domain=2.2241695302319906:4.0742261163904425,variable=\t]({1.*1.2752403180946958*cos(\t r)+0.*1.2752403180946958*sin(\t r)},{0.*1.2752403180946958*cos(\t r)+1.*1.2752403180946958*sin(\t r)});
\draw (39.52735981276525,4.727328208251727)-- (39.52735981276525,3.7273282082517265);
\draw (43.52735981276525,2.7273282082517265)-- (43.52735981276525,1.727328208251727);
\draw [shift={(42.712744599956146,2.7396173657468648)}] plot[domain=-0.8931804282524105:0.9083907676952563,variable=\t]({1.*1.2993564881595177*cos(\t r)+0.*1.2993564881595177*sin(\t r)},{0.*1.2993564881595177*cos(\t r)+1.*1.2993564881595177*sin(\t r)});
\draw (43.511868914161944,3.7641795228497616)-- (42.52735981276525,4.727328208251727);
\draw (42.52735981276525,6.727328208251727)-- (42.52735981276525,5.727328208251727);
\draw (42.52735981276525,3.7273282082517265)-- (42.52735981276525,2.7273282082517265);
\draw [shift={(41.77735981276525,3.727328208251728)}] plot[domain=-0.9272952180016132:0.9272952180016119,variable=\t]({1.*1.25*cos(\t r)+0.*1.25*sin(\t r)},{0.*1.25*cos(\t r)+1.*1.25*sin(\t r)});
\draw [shift={(41.77735981276525,5.727328208251727)}] plot[domain=-0.9272952180016123:0.9272952180016122,variable=\t]({1.*1.25*cos(\t r)+0.*1.25*sin(\t r)},{0.*1.25*cos(\t r)+1.*1.25*sin(\t r)});
\draw [shift={(43.27735981276525,4.727328208251728)}] plot[domain=2.2142974355881813:4.068887871591405,variable=\t]({1.*1.25*cos(\t r)+0.*1.25*sin(\t r)},{0.*1.25*cos(\t r)+1.*1.25*sin(\t r)});
\draw (48.02735981276525,2.7273282082517265)-- (47.02735981276525,1.727328208251727);
\draw (48.02735981276525,4.727328208251727)-- (48.02735981276525,3.7273282082517265);
\draw (48.02735981276525,1.727328208251727)-- (48.02735981276525,0.727328208251727);
\draw [shift={(48.78831073075059,3.7273282082517274)}] plot[domain=2.2212692625499013:4.061916044629685,variable=\t]({1.*1.2566010900770106*cos(\t r)+0.*1.2566010900770106*sin(\t r)},{0.*1.2566010900770106*cos(\t r)+1.*1.2566010900770106*sin(\t r)});
\draw [shift={(48.78831073075059,1.7273282082517267)}] plot[domain=2.2212692625499013:4.061916044629685,variable=\t]({1.*1.256601090077011*cos(\t r)+0.*1.256601090077011*sin(\t r)},{0.*1.256601090077011*cos(\t r)+1.*1.256601090077011*sin(\t r)});
\draw [shift={(47.30797492564451,2.7273282082517265)}] plot[domain=-0.9471784979433826:0.9471784979433825,variable=\t]({1.*1.2318744318386181*cos(\t r)+0.*1.2318744318386181*sin(\t r)},{0.*1.2318744318386181*cos(\t r)+1.*1.2318744318386181*sin(\t r)});
\draw (48.02735981276525,0.727328208251727)-- (48.02735981276525,1.727328208251727);
\draw (48.02735981276525,3.7273282082517265)-- (48.02735981276525,4.727328208251727);
\draw [shift={(47.787047091414756,2.7515904342233544)}] plot[domain=2.224169530231986:4.074226116390447,variable=\t]({1.*1.2752403180946912*cos(\t r)+0.*1.2752403180946912*sin(\t r)},{0.*1.2752403180946912*cos(\t r)+1.*1.2752403180946912*sin(\t r)});
\draw (46.02735981276525,4.727328208251727)-- (47.011868914161944,3.7641795228497616);
\draw (46.02735981276525,3.7273282082517265)-- (46.02735981276525,2.7273282082517265);
\draw (46.02735981276525,5.727328208251727)-- (46.02735981276525,6.727328208251727);
\draw [shift={(45.27735981276525,5.727328208251727)}] plot[domain=-0.9272952180016123:0.9272952180016122,variable=\t]({1.*1.25*cos(\t r)+0.*1.25*sin(\t r)},{0.*1.25*cos(\t r)+1.*1.25*sin(\t r)});
\draw [shift={(45.27735981276525,3.727328208251728)}] plot[domain=-0.9272952180016132:0.9272952180016119,variable=\t]({1.*1.25*cos(\t r)+0.*1.25*sin(\t r)},{0.*1.25*cos(\t r)+1.*1.25*sin(\t r)});
\draw [shift={(46.77735981276525,4.727328208251728)}] plot[domain=2.2142974355881813:4.068887871591405,variable=\t]({1.*1.25*cos(\t r)+0.*1.25*sin(\t r)},{0.*1.25*cos(\t r)+1.*1.25*sin(\t r)});
\draw [shift={(41.303801629353885,-3.6904768936536914)}] plot[domain=2.221269262549897:4.061916044629689,variable=\t]({1.*1.256601090077007*cos(\t r)+0.*1.256601090077007*sin(\t r)},{0.*1.256601090077007*cos(\t r)+1.*1.256601090077007*sin(\t r)});
\draw [shift={(41.303801629353885,-1.6904768936536911)}] plot[domain=2.221269262549897:4.061916044629689,variable=\t]({1.*1.2566010900770068*cos(\t r)+0.*1.2566010900770068*sin(\t r)},{0.*1.2566010900770068*cos(\t r)+1.*1.2566010900770068*sin(\t r)});
\draw [shift={(39.82346582424781,-2.6904768936536922)}] plot[domain=-0.9471784979433817:0.9471784979433829,variable=\t]({1.*1.2318744318386192*cos(\t r)+0.*1.2318744318386192*sin(\t r)},{0.*1.2318744318386192*cos(\t r)+1.*1.2318744318386192*sin(\t r)});
\draw [shift={(36.72823549855945,-2.7027660511488283)}] plot[domain=-0.9083907676952574:0.8931804282524103,variable=\t]({1.*1.2993564881595165*cos(\t r)+0.*1.2993564881595165*sin(\t r)},{0.*1.2993564881595165*cos(\t r)+1.*1.2993564881595165*sin(\t r)});
\draw [shift={(40.30253799001807,-2.7147391196253197)}] plot[domain=2.2089591907891437:4.059015776947596,variable=\t]({1.*1.2752403180946958*cos(\t r)+0.*1.2752403180946958*sin(\t r)},{0.*1.2752403180946958*cos(\t r)+1.*1.2752403180946958*sin(\t r)});
\draw [shift={(42.712744599956146,-2.7396173657468648)}] plot[domain=-0.9083907676952565:0.8931804282524107,variable=\t]({1.*1.2993564881595177*cos(\t r)+0.*1.2993564881595177*sin(\t r)},{0.*1.2993564881595177*cos(\t r)+1.*1.2993564881595177*sin(\t r)});
\draw [shift={(41.77735981276525,-3.727328208251728)}] plot[domain=-0.9272952180016123:0.9272952180016129,variable=\t]({1.*1.25*cos(\t r)+0.*1.25*sin(\t r)},{0.*1.25*cos(\t r)+1.*1.25*sin(\t r)});
\draw [shift={(41.77735981276525,-5.727328208251727)}] plot[domain=-0.9272952180016123:0.9272952180016122,variable=\t]({1.*1.25*cos(\t r)+0.*1.25*sin(\t r)},{0.*1.25*cos(\t r)+1.*1.25*sin(\t r)});
\draw [shift={(43.27735981276525,-4.727328208251728)}] plot[domain=2.2142974355881804:4.068887871591405,variable=\t]({1.*1.25*cos(\t r)+0.*1.25*sin(\t r)},{0.*1.25*cos(\t r)+1.*1.25*sin(\t r)});
\draw [shift={(48.78831073075059,-3.7273282082517274)}] plot[domain=2.221269262549901:4.061916044629685,variable=\t]({1.*1.2566010900770106*cos(\t r)+0.*1.2566010900770106*sin(\t r)},{0.*1.2566010900770106*cos(\t r)+1.*1.2566010900770106*sin(\t r)});
\draw [shift={(48.78831073075059,-1.7273282082517267)}] plot[domain=2.2212692625499013:4.061916044629685,variable=\t]({1.*1.256601090077011*cos(\t r)+0.*1.256601090077011*sin(\t r)},{0.*1.256601090077011*cos(\t r)+1.*1.256601090077011*sin(\t r)});
\draw [shift={(47.30797492564451,-2.7273282082517265)}] plot[domain=-0.9471784979433826:0.9471784979433824,variable=\t]({1.*1.2318744318386181*cos(\t r)+0.*1.2318744318386181*sin(\t r)},{0.*1.2318744318386181*cos(\t r)+1.*1.2318744318386181*sin(\t r)});
\draw [shift={(47.787047091414756,-2.7515904342233544)}] plot[domain=2.2089591907891393:4.0590157769476,variable=\t]({1.*1.2752403180946912*cos(\t r)+0.*1.2752403180946912*sin(\t r)},{0.*1.2752403180946912*cos(\t r)+1.*1.2752403180946912*sin(\t r)});
\draw [shift={(45.27735981276525,-5.727328208251727)}] plot[domain=-0.9272952180016123:0.9272952180016122,variable=\t]({1.*1.25*cos(\t r)+0.*1.25*sin(\t r)},{0.*1.25*cos(\t r)+1.*1.25*sin(\t r)});
\draw [shift={(45.27735981276525,-3.727328208251728)}] plot[domain=-0.9272952180016123:0.9272952180016129,variable=\t]({1.*1.25*cos(\t r)+0.*1.25*sin(\t r)},{0.*1.25*cos(\t r)+1.*1.25*sin(\t r)});
\draw [shift={(46.77735981276525,-4.727328208251728)}] plot[domain=2.2142974355881804:4.068887871591405,variable=\t]({1.*1.25*cos(\t r)+0.*1.25*sin(\t r)},{0.*1.25*cos(\t r)+1.*1.25*sin(\t r)});
\draw (37.54285071136855,-2.6904768936536914)-- (37.54285071136855,-1.6904768936536914);
\draw (40.54285071136855,-2.6904768936536914)-- (39.54285071136855,-1.6904768936536914);
\draw (40.54285071136855,-4.690476893653692)-- (40.54285071136855,-3.6904768936536914);
\draw (40.54285071136855,-1.6904768936536914)-- (40.54285071136855,-0.6904768936536914);
\draw (40.54285071136855,-0.6904768936536914)-- (40.54285071136855,-1.6904768936536914);
\draw (40.54285071136855,-3.6904768936536914)-- (40.54285071136855,-4.690476893653692);
\draw (37.52735981276525,-3.7273282082517265)-- (37.52735981276525,-4.727328208251727);
\draw (39.52735981276525,-4.727328208251727)-- (39.52735981276525,-3.7273282082517265);
\draw (43.52735981276525,-2.7273282082517265)-- (43.52735981276525,-1.727328208251727);
\draw (43.511868914161944,-3.7641795228497616)-- (42.52735981276525,-4.727328208251727);
\draw (42.52735981276525,-6.727328208251727)-- (42.52735981276525,-5.727328208251727);
\draw (42.52735981276525,-3.7273282082517265)-- (42.52735981276525,-2.7273282082517265);
\draw (48.02735981276525,-2.7273282082517265)-- (47.02735981276525,-1.727328208251727);
\draw (48.02735981276525,-4.727328208251727)-- (48.02735981276525,-3.7273282082517265);
\draw (48.02735981276525,-1.727328208251727)-- (48.02735981276525,-0.727328208251727);
\draw (48.02735981276525,-0.727328208251727)-- (48.02735981276525,-1.727328208251727);
\draw (48.02735981276525,-3.7273282082517265)-- (48.02735981276525,-4.727328208251727);
\draw (46.02735981276525,-4.727328208251727)-- (47.011868914161944,-3.7641795228497616);
\draw (46.02735981276525,-3.7273282082517265)-- (46.02735981276525,-2.7273282082517265);
\draw (46.02735981276525,-5.727328208251727)-- (46.02735981276525,-6.727328208251727);
\begin{scriptsize}
\draw [fill=qqqqff] (37.54285071136855,2.6904768936536914) circle (2.5pt);
\draw [fill=qqqqff] (37.54285071136855,1.6904768936536914) circle (2.5pt);
\draw [fill=qqqqff] (40.54285071136855,2.6904768936536914) circle (2.5pt);
\draw [fill=qqqqff] (39.54285071136855,1.6904768936536914) circle (2.5pt);
\draw [fill=qqqqff] (40.54285071136855,4.690476893653692) circle (2.5pt);
\draw [fill=qqqqff] (40.54285071136855,0.6904768936536914) circle (2.5pt);
\draw [fill=qqqqff] (40.54285071136855,1.6904768936536914) circle (2.5pt);
\draw [fill=qqqqff] (40.54285071136855,3.6904768936536914) circle (2.5pt);
\draw [fill=qqqqff] (37.52735981276525,3.7273282082517265) circle (2.5pt);
\draw [fill=qqqqff] (37.52735981276525,4.727328208251727) circle (2.5pt);
\draw [fill=qqqqff] (39.52735981276525,3.7273282082517265) circle (2.5pt);
\draw [fill=qqqqff] (39.52735981276525,4.727328208251727) circle (2.5pt);
\draw [fill=qqqqff] (43.52735981276525,2.7273282082517265) circle (2.5pt);
\draw [fill=qqqqff] (43.511868914161944,3.7641795228497616) circle (2.5pt);
\draw [fill=qqqqff] (42.52735981276525,4.727328208251727) circle (2.5pt);
\draw [fill=qqqqff] (42.52735981276525,5.727328208251727) circle (2.5pt);
\draw [fill=qqqqff] (42.52735981276525,6.727328208251727) circle (2.5pt);
\draw [fill=qqqqff] (42.52735981276525,3.7273282082517265) circle (2.5pt);
\draw [fill=qqqqff] (42.52735981276525,2.7273282082517265) circle (2.5pt);
\draw [fill=qqqqff] (48.02735981276525,2.7273282082517265) circle (2.5pt);
\draw [fill=qqqqff] (47.02735981276525,1.727328208251727) circle (2.5pt);
\draw [fill=qqqqff] (48.02735981276525,4.727328208251727) circle (2.5pt);
\draw [fill=qqqqff] (48.02735981276525,0.727328208251727) circle (2.5pt);
\draw [fill=qqqqff] (48.02735981276525,1.727328208251727) circle (2.5pt);
\draw [fill=qqqqff] (48.02735981276525,3.7273282082517265) circle (2.5pt);
\draw [fill=qqqqff] (47.011868914161944,3.7641795228497616) circle (2.5pt);
\draw [fill=qqqqff] (46.02735981276525,4.727328208251727) circle (2.5pt);
\draw [fill=qqqqff] (46.02735981276525,5.727328208251727) circle (2.5pt);
\draw [fill=qqqqff] (46.02735981276525,6.727328208251727) circle (2.5pt);
\draw [fill=qqqqff] (46.02735981276525,3.7273282082517265) circle (2.5pt);
\draw [fill=qqqqff] (46.02735981276525,2.7273282082517265) circle (2.5pt);
\draw [fill=qqqqff] (40.54285071136855,-4.690476893653692) circle (2.5pt);
\draw [fill=qqqqff] (40.54285071136855,-2.6904768936536914) circle (2.5pt);
\draw [fill=qqqqff] (40.54285071136855,-2.6904768936536914) circle (2.5pt);
\draw [fill=qqqqff] (40.54285071136855,-0.6904768936536914) circle (2.5pt);
\draw [fill=qqqqff] (40.54285071136855,-1.6904768936536914) circle (2.5pt);
\draw [fill=qqqqff] (40.54285071136855,-3.6904768936536914) circle (2.5pt);
\draw [fill=qqqqff] (37.54285071136855,-1.6904768936536914) circle (2.5pt);
\draw [fill=qqqqff] (37.52735981276525,-3.7273282082517265) circle (2.5pt);
\draw [fill=qqqqff] (39.54285071136855,-1.6904768936536914) circle (2.5pt);
\draw [fill=qqqqff] (39.52735981276525,-3.7273282082517265) circle (2.5pt);
\draw [fill=qqqqff] (43.52735981276525,-1.727328208251727) circle (2.5pt);
\draw [fill=qqqqff] (43.511868914161944,-3.7641795228497616) circle (2.5pt);
\draw [fill=qqqqff] (42.52735981276525,-2.7273282082517265) circle (2.5pt);
\draw [fill=qqqqff] (42.52735981276525,-4.727328208251727) circle (2.5pt);
\draw [fill=qqqqff] (42.52735981276525,-4.727328208251727) circle (2.5pt);
\draw [fill=qqqqff] (42.52735981276525,-6.727328208251727) circle (2.5pt);
\draw [fill=qqqqff] (42.52735981276525,-5.727328208251727) circle (2.5pt);
\draw [fill=qqqqff] (42.52735981276525,-3.7273282082517265) circle (2.5pt);
\draw [fill=qqqqff] (48.02735981276525,-4.727328208251727) circle (2.5pt);
\draw [fill=qqqqff] (48.02735981276525,-2.7273282082517265) circle (2.5pt);
\draw [fill=qqqqff] (48.02735981276525,-2.7273282082517265) circle (2.5pt);
\draw [fill=qqqqff] (48.02735981276525,-0.727328208251727) circle (2.5pt);
\draw [fill=qqqqff] (48.02735981276525,-1.727328208251727) circle (2.5pt);
\draw [fill=qqqqff] (48.02735981276525,-3.7273282082517265) circle (2.5pt);
\draw [fill=qqqqff] (47.02735981276525,-1.727328208251727) circle (2.5pt);
\draw [fill=qqqqff] (47.011868914161944,-3.7641795228497616) circle (2.5pt);
\draw [fill=qqqqff] (46.02735981276525,-6.727328208251727) circle (2.5pt);
\draw [fill=qqqqff] (46.02735981276525,-4.727328208251727) circle (2.5pt);
\draw [fill=qqqqff] (46.02735981276525,-4.727328208251727) circle (2.5pt);
\draw [fill=qqqqff] (46.02735981276525,-2.7273282082517265) circle (2.5pt);
\draw [fill=qqqqff] (46.02735981276525,-3.7273282082517265) circle (2.5pt);
\draw [fill=qqqqff] (46.02735981276525,-5.727328208251727) circle (2.5pt);
\draw [fill=qqqqff] (37.54285071136855,-2.6904768936536914) circle (2.5pt);
\draw [fill=qqqqff] (37.54285071136855,-1.6904768936536914) circle (2.5pt);
\draw [fill=qqqqff] (40.54285071136855,-2.6904768936536914) circle (2.5pt);
\draw [fill=qqqqff] (39.54285071136855,-1.6904768936536914) circle (2.5pt);
\draw [fill=qqqqff] (40.54285071136855,-4.690476893653692) circle (2.5pt);
\draw [fill=qqqqff] (40.54285071136855,-3.6904768936536914) circle (2.5pt);
\draw [fill=qqqqff] (40.54285071136855,-1.6904768936536914) circle (2.5pt);
\draw [fill=qqqqff] (40.54285071136855,-0.6904768936536914) circle (2.5pt);
\draw [fill=qqqqff] (40.54285071136855,-0.6904768936536914) circle (2.5pt);
\draw [fill=qqqqff] (40.54285071136855,-1.6904768936536914) circle (2.5pt);
\draw [fill=qqqqff] (40.54285071136855,-3.6904768936536914) circle (2.5pt);
\draw [fill=qqqqff] (40.54285071136855,-4.690476893653692) circle (2.5pt);
\draw [fill=qqqqff] (37.52735981276525,-3.7273282082517265) circle (2.5pt);
\draw [fill=qqqqff] (37.52735981276525,-4.727328208251727) circle (2.5pt);
\draw [fill=qqqqff] (39.52735981276525,-4.727328208251727) circle (2.5pt);
\draw [fill=qqqqff] (39.52735981276525,-3.7273282082517265) circle (2.5pt);
\draw [fill=qqqqff] (43.52735981276525,-2.7273282082517265) circle (2.5pt);
\draw [fill=qqqqff] (43.52735981276525,-1.727328208251727) circle (2.5pt);
\draw [fill=qqqqff] (43.511868914161944,-3.7641795228497616) circle (2.5pt);
\draw [fill=qqqqff] (42.52735981276525,-4.727328208251727) circle (2.5pt);
\draw [fill=qqqqff] (42.52735981276525,-6.727328208251727) circle (2.5pt);
\draw [fill=qqqqff] (42.52735981276525,-5.727328208251727) circle (2.5pt);
\draw [fill=qqqqff] (42.52735981276525,-3.7273282082517265) circle (2.5pt);
\draw [fill=qqqqff] (42.52735981276525,-2.7273282082517265) circle (2.5pt);
\draw [fill=qqqqff] (48.02735981276525,-2.7273282082517265) circle (2.5pt);
\draw [fill=qqqqff] (47.02735981276525,-1.727328208251727) circle (2.5pt);
\draw [fill=qqqqff] (48.02735981276525,-4.727328208251727) circle (2.5pt);
\draw [fill=qqqqff] (48.02735981276525,-3.7273282082517265) circle (2.5pt);
\draw [fill=qqqqff] (48.02735981276525,-1.727328208251727) circle (2.5pt);
\draw [fill=qqqqff] (48.02735981276525,-0.727328208251727) circle (2.5pt);
\draw [fill=qqqqff] (48.02735981276525,-0.727328208251727) circle (2.5pt);
\draw [fill=qqqqff] (48.02735981276525,-1.727328208251727) circle (2.5pt);
\draw [fill=qqqqff] (48.02735981276525,-3.7273282082517265) circle (2.5pt);
\draw [fill=qqqqff] (48.02735981276525,-4.727328208251727) circle (2.5pt);
\draw [fill=qqqqff] (46.02735981276525,-4.727328208251727) circle (2.5pt);
\draw [fill=qqqqff] (47.011868914161944,-3.7641795228497616) circle (2.5pt);
\draw [fill=qqqqff] (46.02735981276525,-3.7273282082517265) circle (2.5pt);
\draw [fill=qqqqff] (46.02735981276525,-2.7273282082517265) circle (2.5pt);
\draw [fill=qqqqff] (46.02735981276525,-5.727328208251727) circle (2.5pt);
\draw [fill=qqqqff] (46.02735981276525,-6.727328208251727) circle (2.5pt);
\end{scriptsize}
\end{tikzpicture}
\caption{$8$ different lifts of $N$} \label{fig:8liftsofn}
\end{figure}
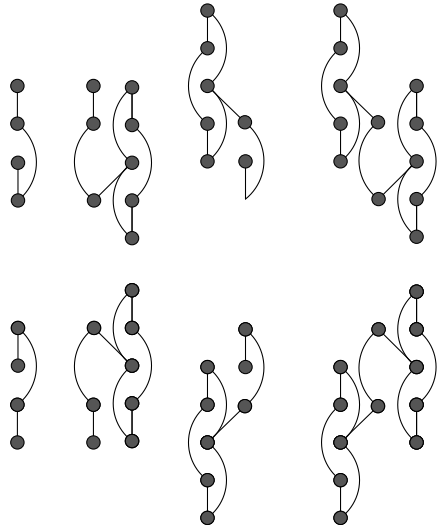

\bibliographystyle{alpha}
\bibliography{biblio}

\end{document}